\documentclass[11pt]{article}

\usepackage{amssymb}
\usepackage{amsmath}
\usepackage{epsfig}
\usepackage{graphicx}
\usepackage[all]{xy}
\usepackage{color}
\usepackage{tikz, calc}
\usepackage{pgfplots}
\usepackage{enumitem}
\setitemize{label=\scriptsize{$\blacksquare$}}

\usepackage{concmath}
\usepackage[T1]{fontenc}

\usepackage[binary-units]{siunitx}

\usepackage{authblk}
\usepackage[numbers,sort&compress]{natbib}

\usepackage[a4paper, left=3.2cm,top=3cm,right=3.2cm,bottom=3.5cm]{geometry}
\usepackage{soul}
\usepackage{xcolor}
\colorlet{lred}{red!40}
\colorlet{lgreen}{green!40}
\colorlet{lblue}{blue!40}

\newcommand{\rot}{}

\usetikzlibrary{calc}

\usepackage{setspace}
\usepackage{amsthm,amsfonts,amssymb,epsfig,graphics,amsmath,amsbsy,subfigure}

\newcommand{\half}{^\infty_0 }

\newcommand{\R}{\mathbb{R}}
\newcommand{\sph}{\mathbb{S}}

\newcommand{\NN}{\mathbb{N}}

\newcommand{\xx}{\mathbf{x}}

\newcommand{\ttheta}{\boldsymbol{\theta}}

\newcommand{\oomega}{\boldsymbol{\omega}}
\newcommand{\xxi}{\boldsymbol{\xi}}

\newcommand{\ca}{c_1}
\newcommand{\cb}{c_2}

\newcommand{\mf}{\mathtt{ \mathbf{m} }}
\newcommand{\nf}{\mathtt{ \mathbf{n} }}

\newcommand{\Sin}{\mathtt{S}  }
\newcommand{\Cos}{\mathtt{C}  }
\newcommand{\gdata}{\mathtt{g}  }

\newcommand{\Ntime}{\mathtt{N}_t}
\newcommand{\Nrad}{\mathtt{N}_r}
\newcommand{\Nangle}{\mathtt{ {N_\theta} }}

\newcommand{\Br}{B_R(0)}
\newcommand{\Sr}{S_R}

\newcommand{\Mo}{\mathbf{M}}
\newcommand{\So}{\mathbf{S}}
\newcommand{\CCo}{\mathbf{C}}

\newcommand{\kl}[1]{(#1)}
\newcommand{\bkl}[1]{\left(#1\right)}
\newcommand{\abs}[1]{\lvert#1\rvert}
\newcommand{\norm}[1]{\lVert#1\rVert}
\newcommand{\set}[1]{\{#1\}}

\makeatletter
\newcommand*\bigcdot{\mathpalette\bigcdot@{.6}}
\newcommand*\bigcdot@[2]{\mathbin{\vcenter{\hbox{\scalebox{#2}{$\m@th#1\bullet$}}}}}
\makeatother
\newcommand\inner[2]{{#1}\bigcdot {#2}}

\newcommand{\rmd}[1]{\mathrm d#1}

\newtheorem{theorem}{Theorem}

\newtheorem{cor}{Corollary}
\newtheorem{lem}{Lemma}

\numberwithin{equation}{section}
\numberwithin{figure}{section}
\numberwithin{table}{section}
\allowdisplaybreaks

\title{Photoacoustic Tomography with Direction Dependent Data:
An Exact Series Reconstruction Approach}

\date{}

\author{Gerhard Zangerl}

 \affil{Department of Mathematics, University of Innsbruck\authorcr
Technikerstrasse 13, 6020 Innsbruck, Austria
 \authorcr E-mail:  \texttt{gerhard.zangerl@uibk.ac.at} }

\author{Sunghwan  Moon}

 \affil{Department of Mathematics, College of Natural Sciences\authorcr
Kyungpook National University, Daegu 41566, Republic of Korea
 \authorcr E-mail:  \texttt{sunghwan.moon@knu.ac.kr} }

\author{Markus Haltmeier}

 \affil{Department of Mathematics, University of Innsbruck\authorcr
Technikerstrasse 13, 6020 Innsbruck, Austria
 \authorcr E-mail:  \texttt{markus.haltmeier@uibk.ac.at} }

\begin{document}

\maketitle

%\begin{frontmatter}
\begin{abstract}
Photoacoustic image reconstruction often assumes that the
restriction of the acoustic pressure on the  detection surface is given.
However, commonly used detectors often have a certain directivity and frequency dependence, in which case the measured data are more  accurately  described as
a linear combination of the acoustic pressure and its normal derivative on  the detection surface.
In this paper, we consider the inverse source problem for data that are a combination of an acoustic pressure of the wave equation and its normal derivative
For the special case of a spherical detection geometry we  derive
exact frequency domain reconstruction formulas. We present
numerical results showing the robustness and validity of the derived formulas.
Moreover, we compare several different combinations of the pressure and its normal derivative showing that used measurement  model significantly
affects the recovered initial pressure.


\end{abstract}

\section{Introduction}
\label{sec:intro}

Photoacoustic Tomography (PAT) is a hybrid imaging technique that combines high optical contrast with good ultrasonic resolution. It is based on the generation of an ultrasound wave  inside an object of interest by pulses optical illumination. The initial pressure distribution of the induced sound wave encodes the optical absorption properties of the object, which are of great interest in medical diagnostics. PAT has proven to be very promising for medical applications like functional brain imaging of small animals, early cancer diagnostics, and imaging of vasculature \cite{beard2011biomedical,wang2009multiscale}.

\begin{figure}[htb!]
	\begin{center}	
	\begin{tikzpicture}
%	\path [draw=none,fill= green!50!black, fill opacity = 0.4,even odd rule] (1,0) circle (2) (1,0) circle (3);
	\foreach \a in {1,2,...,17}
	{
    \draw[->, -latex , thick] (\a *360/17:2cm)--(\a *360/17:3cm);	
    }
	\draw[thick] (0,0) circle (2cm);
	\foreach \a in {1,2,...,17}
	{
		 \draw[fill = red] (\a*360/17: 2cm) circle (2 pt);
		 	}
	\draw[fill = green!50!black, line width = 0pt] (1,0) circle (0.5cm);
	\draw[thick, color = green!50!black] (1,0) circle (2);
	\draw[thick, color = green!50!black] (1,0) circle (2.2);
	\draw[thick, color = green!50!black] (1,0) circle (2.4);
	\draw[thick, color = green!50!black] (1,0) circle (2.6);
	\draw[thick,, color = green!50!black] (1,0) circle (2.8);
	\draw[thick, color = green!50!black] (1,0) circle (3);
	\end{tikzpicture}
    \end{center}\caption{Pressure wave  (green circles) emitted by an
    a disk.  Each detector is displayed as a red dot on the detection surface,
    which measures a combination of  the  pressure  and its  normal derivative.}
    \label{fig:setup}
\end{figure}
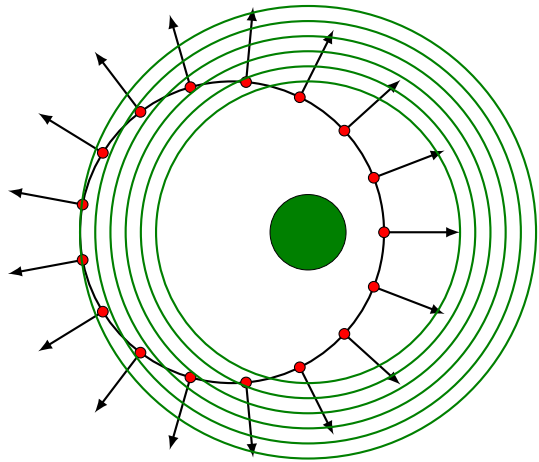

In a typical PAT setup, the induced acoustic waves  are recorded by several point-like detectors outside the  support of the  investigates sample.  Typically, the detectors  are located on a surface $S$ that  fully or partially  encloses the volume $\Omega$ in which the object of interest is contained; compare Figure \ref{fig:setup}.
In most reconstruction approaches, the measured  data are identified with the  restriction of the acoustic wave to the surface $\partial \Omega$ or sampled values of it, possible convolved in time with the detector impulse response function.
{\rot Typical piezoelectric sensors,  which are commonly used in PAT, however, have a
directional  dependence and are most sensitive in the normal direction to the  sensor surface. Moreover, at the resonance frequency they are more sensitive than at lower
frequencies. As noted in  \cite{xu2004time}, such measurement data are rather  modeled by the  combination of the pressure field and its normal derivative, than the  pressure alone. This is also suggested by visual comparison with real data \cite{roitner2014deblurring,paltauf2017piezoelectric}. Systematic theoretical and experimental studies  on PAT sensor modeling  are interesting lines of
future research.}

{\rot Let $f$ be the smooth compactly supported initial pressure,
and $p \colon \R^n \times (0, \infty) \to \R$
the induced pressure wave that satisfies
\begin{alignat}{2}\label{eq:wavsys1}
\partial^2_t p(\xx,t) - \Delta_\xx p(\xx,t) &= 0 \qquad &&(\xx,t)\in\R^n\times (0,\infty)\nonumber \\
p(\xx,0)   			&=f(\xx) \qquad          &&\xx\in\R^n   \\
\partial_tp(\xx,0)  &=0       \qquad         &&\xx\in\R^n \,. \nonumber
\end{alignat}
In order to incorporate directional dependence, we model the data  of a detector located at $\xx \in \partial \Omega$ by
\begin{equation}\label{eq:dirdata}
m_{\ca, \cb}(\xx, t) =  \ca p(\xx,  t)  +  \cb \inner{n(\xx)}{ \nabla p(\xx, t)}   \quad \text{ for } (\xx, t) \in \partial \Omega \times [0,T] \,,
\end{equation}}
where  $\inner{n(\xx)}{ \nabla p(\xx, t)}$ is the normal derivative of the pressure, $n(\xx)$ the outwards pointing normal of the surface at $\xx$,
$T$ the final measurement time,
and $\ca, \cb \in \R$ are constants.  {\rot The goal is to recover the initial data $f$ from data in \eqref{eq:dirdata}.}
To the best of our knowledge, this
paper is the first work investigating PAT with direction dependent data of the form~\eqref{eq:dirdata} allowing arbitrary values of $\ca,\cb$.

The standard  image reconstruction problem in PAT corresponds to the special case
$\ca \neq 0$ and $\cb =0$ in our data model \eqref{eq:dirdata}. Many methods for reconstructing the initial pressure distribution have been derived in the recent years in various situations. This includes, for example, different detection geometries with variable or constant speed of sound, or limited view situations. Theoretical questions concerning uniqueness and stability of the inverse source problem have also been  investigated \cite{haltmeier2017analysis,kuchment2008mathematics,kunyansky2008thermo,belhachmi2016direct,stefanov2009thermo}. A practically important case assumes a constant speed of sound. In this case, several exact reconstruction formulas have been developed \cite{kunyansky2011reconstruction,kunyansky2015inversion,Haltmeier14,Halt2d,nguyen2009,finch04det,natterer2012photo,palamodov2012,haltmeier2009frequency,agranovsky2007uniqueness,FinHalRak07,zangerl2009exact,kunyansky2007series,haltmeier2007thermoacoustic}. Among these formulas so-called series expansion formulas provide very fast and accurate reconstructions. They give an expansion of the initial pressure in terms of eigenfunctions of the Laplacian \cite{haltmeier2009frequency,agranovsky2007uniqueness,zangerl2009exact,kunyansky2007series,haltmeier2007thermoacoustic}.

The case where measurements are modeled  by the normal derivative of the pressure
($\ca=0$ and $\cb \neq 0$ in \eqref{eq:dirdata})
 is studied much less. It has been considered in \cite{Finch2005,finch2007spherical}, where an  explicit  inversion formula of the backprojection type  is derived  for the case that the detection surface is a sphere in three spatial  dimensions.
 Besides that, we are not aware  of any results for the case $\cb \neq 0$. We are not aware of any results when $\ca,\cb$ are both non-vanishing.

In this paper, we provide series expansion reconstruction formulas for the direction dependent data model \eqref{eq:dirdata} allowing arbitrary values of $\ca, \cb \in \R$, for
the case that the measurement surface is a sphere. Our approach
is  based on expansions in spherical harmonics and an  explicit formula  relating the  spherical harmonics coefficients of the direction dependent data, as a function of time, and the Fourier coefficients of the initial pressure distribution $f$, as a function of distance to the origin
(see Section \ref{sec:pat}).
By using Fourier Bessel series in the radial variable,
in Section \ref{sec:inv} we derive inversion formulas that allows for a
 stable implementation. Numerical implementation and results
 are presented in Section \ref{sec:num}. The paper ends with a conclusion
 and outlook presented in Section~\ref{sec:conclusion}

\section{Preliminaries}
\label{sec:pat}

In PAT, sound propagation  is commonly described by an acoustic
pressure wave $p \colon \R^n \times [0, \infty) \to \R$ that
satisfies the initial value problem\eqref{eq:wavsys1}.
In this work, we assume that the initial pressure distribution satisfies  $f \in \mathcal{C}_0^\infty \left( \Omega   \right)$. In PAT with direction dependent data we model measurement data   by \eqref{eq:dirdata}, where  $p$ is the solution of \eqref{eq:wavsys1}.
The aim  is to recover the initial pressure distribution $f$ from such data. To the best of our knowledge, this is a new inverse  problem for the wave equation that has not been considered so far. For the  practical application, the cases $n =2$ and $n =3$ are of relevance. In particular, the case $n =2$ appears from direction dependent measurements with integrating line detectors \cite{burgholzer2007temporal,haltmeier2007thermoacoustic}.

In particular, we study the case where $\Omega  = \Br$ is the ball of radius $R$ centered at the origin.
In this  situation, we can write the measurement data in the form
\begin{align}\label{eq:data}
g (\ttheta, t) :=   \ca p(R\ttheta , t) +  \cb   \Bigl[\inner {\ttheta}{  \nabla} p (\xx, t) \bigr]_{\xx = R \ttheta}
\quad \text { for } (\ttheta, t ) \in \sph^{n-1} \times [0,T]\,.
\end{align}
We write  $\Mo_{\ca,\cb}$ for the  operator  that takes the initial data $f$ in \eqref{eq:wavsys1}
to the corresponding boundary data.  The constants $\ca, \cb \in \R$ are weights describing the  contribution of the pressure and its  normal derivative, respectively, to the measures data.
The operator $\Mo_{1,0}$  is the standard PAT forward operator, and the operator $\Mo_{0,1}$  corresponds to the case where the  data only consists of the normal derivative.
Clearly, we have $\Mo_{\ca,\cb} =  \ca \Mo_{1,0} +
\cb \Mo_{0,1}$.

For the derivation of the inversion formulas we  use $T = \infty$. Because of the strict Huygens principle,  in the case of odd spatial dimension, we have $p(\xx, t) = 0$ for $t> 2R$, which
 yields exact inversion for any measurement  time  $T \geq 2R$. For even spatial dimension,
 this is not the case and, strictly taken, our inversion formulas  are exact only for $T =  \infty$. Deriving exact formulas using data over a finite time  interval $[0,T]$
 in even dimension is an interesting open issue. For the standard PAT operator $\Mo_{1,0}$  such a formula has been derived in \cite[Theorem 1.4]{FinHalRak07}.

\subsection{Notation}\label{sec:notation}

Our approach is based on Fourier methods that lead to a relation between
measurement data and the initial pressure distribution in the frequency domain.
We denote the involved transforms by
\begin{align*}
   \hat f   (\xxi)  &: =  \int_{\R^n}  f(\xx) e^{-  i \inner{ \xx}{\xxi}  } {\rm d}\xx,
&& \quad  \text{ for }  f \in L^{1}\left( \R^n \right), \,  \xxi \in \R^n     \\
  \CCo  \left\{ g \right\}   (\lambda)
   &:=      \int_0^\infty g(t) \cos(\lambda t) {\rm d}t,
&& \quad  \text{ for }  g \in L^{1}(0,\infty), \,  \lambda > 0\\
  \So    \left\{ g \right\}     (\lambda)
  &:=      \int_0^\infty g(t) \sin(\lambda t) {\rm d}t,
&& \quad  \text{ for }  g \in L^{1}(0,\infty), \,  \lambda > 0,
\end{align*}
which are the Fourier, cosine and sine transforms, respectively.

Moreover, we denote by $Y_{\ell,k}(\ttheta)$ the spherical harmonics \cite{natterer1986math},
which form a complete orthonormal system in $L^2(\sph^{n-1})$. In particular,
$$
\forall f \in L^2(\R^n)
\colon \quad
f(\rho  \ttheta )
=   \sum_{\ell = 0}^\infty  \sum_{k=0}^{N(n,\ell)}
f_{\ell,k}(\rho)  \,  Y_{\ell,k}(\ttheta) \,, $$
where   $N(n,\ell)= (2\ell+n-2)(n+\ell-3)!/(\ell!(n-2)!)$ for $\ell \in \NN$
and $N(n,0):=1$. We  write  $J_\nu$ for the $\nu$-th order Bessel function of the first kind and denote by $w_{j,\ell}$ the $j$-th positive root
of $J_{\ell+(n-2)/2}$ .

\subsection{Relations between transform coefficients}\label{sec:relations}

For the following we use the spherical harmonics expansions
of the measurement data  $g = \Mo_{\ca,\cb} f $  in \eqref{eq:data} and the
Fourier transform $\hat f$ of the initial pressure,
\begin{align}\label{eq:e1}
&\forall   (t,\ttheta)\in (0,\infty)\times \sph^{n-1} \colon
\quad
g(\ttheta,t)
= \sum_{\ell=0}^\infty\sum^{N(n,\ell)}_{k=0}g_{\ell,k}(t)Y_{\ell,k}(\ttheta)
\\ \label{eq:e2}
& \forall (\lambda,\oomega)\in (0,\infty)\times \sph^{n-1}  \colon
\quad
\Hat{f}(\lambda\oomega)
=\sum_{\ell=0}^\infty\sum^{N(n,\ell)}_{k=0} \Hat{f}_{\ell,k}(\lambda)Y_{\ell,k}(\oomega) \,.
\end{align}

The following Lemma is the key to our results.

\begin{lem}\label{lem:reloffandp}
Let $f\in \mathcal{C}^\infty_0(\R^n)$ be the initial pressure distribution in \eqref{eq:wavsys1} and let $g = \Mo_{\ca,\cb} f $ be the
corresponding measurement data given by \eqref{eq:data}. Then,
\begin{multline}\label{eq:cplk}
\forall \lambda>0 \colon \quad \CCo\set{g_{\ell,k}}(\lambda)=
 (2R\pi)^{-{\frac{n}{2}}}\pi 2^{-1} i^\ell\\
\times \left[(R\ca+\cb \ell)J_{\ell+\frac{n-2}{2}}(R\lambda) - \cb R\lambda J_{\ell+\frac{n}{2}}(R\lambda)\right] \,
\hat f_{\ell,k}(\lambda) \lambda^{\frac{n}2}\,.
\end{multline}
\end{lem}

\begin{proof} \rot
Let $p$ be  the solution of \eqref{eq:wavsys1}. Then $p(\xx,t)=(2\pi)^{-n}  \int_{\R^n}\cos(t|\xxi|)e^{i\inner{\xx}{\xxi}} \hat f(\xxi)\rmd\xxi$ and therefore
	%\begin{align*}
	$\inner{\ttheta}{\nabla  p}(\xx,t)=(2\pi)^{-n}\int_{\R^n}({i}\inner{\ttheta}{\xxi})\cos(t|\xxi|)e^{{i}\inner{\xx}{\xxi}}\hat f(\xxi)\rmd\xxi $
for all $(\xx,t)\in\R^n\times\R $.
Changing to spherical coordinates, $\lambda\oomega \gets \xxi$, and using the spherical harmonics expansion \eqref{eq:e1}
we can write the direction dependent measurements as
\begin{align*}
g(\ttheta,t)&= \ca p(R\ttheta  , t) +  \cb   \bigl[\inner {\ttheta}{  \nabla} p (\xx, t) \bigr]_{\xx = R \ttheta}
\\
&= (2\pi)^{-n}   \int_{\sph^{n-1}}\int\half[\ca+\cb(i\inner{\ttheta}{\lambda\oomega})] \cos(t\lambda)e^{i\lambda\inner{R\ttheta}{\oomega}}\hat f(\lambda\oomega)\lambda^{n-1}\rmd\lambda \rmd S(\oomega)
\\
&=(2\pi)^{-n}  \int_{\sph^{n-1}} \int\half \left( \ca e^{i\lambda\inner{R\ttheta}{\oomega}}
+ \cb \lambda
\left(
\partial_{R\lambda} e^{i\lambda\inner{R\ttheta}{\oomega}} \right) \right)     \cos(t\lambda)\hat f(\lambda\oomega)\lambda^{n-1 }\rmd\lambda \rmd S(\oomega)
\\
 &=(2\pi)^{-n}\sum_{\ell=0}^\infty\sum^{N(n,\ell)}_{k=0}\int\half \cos(t\lambda)\hat f_{\ell,k}(\lambda)\lambda^{n-1} \left( \ca  A_{\ell,k}  (\ttheta) + \cb B_{\ell,k}  (\ttheta) \right)\rmd\lambda \,,
\end{align*}
where we used the abbreviations
\begin{align*}
A_{\ell,k}(\ttheta) &=  \int_{\sph^{n-1}}e^{i R\lambda \inner{\ttheta}{\oomega}}Y_{\ell,k}(\oomega) \rmd S(\oomega)
\\
B_{\ell,k}(\ttheta) &=  \lambda \left[ \partial_{R\lambda}   \int_{\sph^{n-1}}e^{iR\lambda\inner{\ttheta}{\oomega}}Y_{\ell,k}(\oomega) \rmd S(\oomega) \right]\,.
\end{align*}
According to the Fuck-Hecke theorem \cite{natterer1986math}
\begin{equation}\label{eq:funkhecke}
A_{\ell,k}(\ttheta)   = (2\pi)^{\frac{n}{2}} i^\ell(R\lambda)^{\frac{2-n}{2}}J_{\ell+\frac{n-2}{2}}(R\lambda)Y_{\ell,k}(\ttheta)\,.
\end{equation} %\ma{From here on correct for $R\lambda$ instead of $\lambda$}
Moreover, using the identity  $\partial_{R\lambda}[(R\lambda)^{-\nu}J_\nu(R \lambda)]= - (R\lambda)^{-\nu}J_{\nu+1}(R\lambda)$ (see, for example, \cite{folland2009fourier}), we further obtain
\begin{align*}%\label{eq:partiallambda}
B_{\ell,k}(\ttheta)  &=
\lambda \, \partial_{R\lambda} \left[    A_{\ell,k}(\ttheta) \right]
\\& = \lambda \,
(2\pi)^{\frac{n}{2}} \, i^\ell \, Y_{\ell,k}(\ttheta)\,
\partial_{R\lambda} \left[ (R\lambda)^{-(n-2)/2}J_{\ell+\frac{n-2}{2}}(R\lambda)\right]
\\& = \lambda \,
(2\pi)^{\frac{n}{2}} \, i^\ell \, Y_{\ell,k}(\ttheta)\,
\partial_{R\lambda} \left[ (R\lambda)^\ell\, (R\lambda)^{-(\ell +(n-2)/2)}J_{\ell+\frac{n-2}{2}}(R\lambda)\right]
\\& = \lambda \,
 (2\pi)^{\frac{n}{2}} \, i^\ell \, Y_{\ell,k}(\ttheta)\,
 \left[ \ell  (R\lambda)^{-n/2 }J_{\ell+\frac{n-2}{2}}(R\lambda)
 - (R\lambda)^{-(n-2)/2}J_{\ell+n/2}(R\lambda)  \right]\\&
 =
(2\pi)^{\frac{n}{2}} \, i^\ell \,
 (R\lambda)^{\frac{2-n}{2}} \,
 \left[ R^{-1} \ell   J_{\ell+\frac{n-2}{2}}(R\lambda)
 -\lambda\, J_{\ell+n/2}(R\lambda)  \right]
 \,
 Y_{\ell,k}(\ttheta)
 \,.
\end{align*}
The last two displayed equations imply $\ca  A_{\ell,k}  (\ttheta) + \cb B_{\ell,k}  (\ttheta) %\\
= (2\pi)^{\frac{n}{2}} i^\ell\, (R\lambda)^{\frac{2-n}{2}}$
$
Y_{\ell,k}(\ttheta) ( (\ca + R^{-1} \cb  \ell)J_{\ell+\frac{n-2}{2}}(R\lambda) - \cb \lambda\, J_{\ell+n/2}(R\lambda) )$
and therefore
  \begin{multline}\label{eq:e3}
g(\ttheta,t)  = R^{-\frac{n}{2}} (2\pi)^{-{\frac{n}{2}}}\sum_{\ell=0}^\infty\sum^{N(n,\ell)}_{k=0} i^\ell\int \half \cos(t\lambda)\hat f_{\ell,k}(\lambda)
(\lambda)^{\frac{n}{2}} \\
\times \left(  (R \ca + \cb \ell)   J_{\ell+\frac{n-2}{2} }(R\lambda)  - \cb R \lambda  J_{\ell+n/2}(R\lambda) \right)  {\rm d}\lambda \;Y_{\ell,k}(\ttheta)  \,.
\end{multline}
Comparing expansions \eqref{eq:e1} and \eqref{eq:e3} we conclude that
\begin{multline*}
g_{\ell,k}(t)=R^{-\frac{n}{2}} (2\pi)^{-{\frac{n}{2}}} i^\ell\int\half \cos(t\lambda)\hat f_{\ell,k}(\lambda) \lambda^{\frac{n}{2}  } \\
\times  \left(  (R\ca +   \cb \ell)   J_{\ell+\frac{n-2}{2} }(R\lambda)  - \cb R \lambda  J_{\ell+n/2}(R\lambda) \right)  \rmd\lambda\,.
\end{multline*}
This becomes \eqref{eq:cplk} after applying the inversion formula
$g =\frac{2}{\pi}\CCo\CCo{g} $ for the cosine transform; see \cite[Equation~(7.28)]{folland2009fourier}.
\end{proof}

As a first application of  Lemma~\ref{lem:reloffandp}, we obtain a range
condition for the  classical  PAT forward operator $\Mo_{1, 0}$
that maps the initial data $f$ to the solution of \eqref{eq:wavsys1}
on the boundary. More precisely, we have the following result.

{\rot \begin{cor}[Range condition]\label{cor:range}
Let $f \in \mathcal{C}_0^\infty(\Br)$. Then $\CCo\{ p_{\ell,k}\}(R^{-1} w_{j,\ell}) = 0$
where $(p_{\ell,k})_{\ell, k}$ are the  spherical harmonics coefficients of   $\Mo_{1, 0} f$ and $w_{j,\ell}$ is the $j$-th positive zero of $J_{\ell+(n-2)/2}$.
\end{cor}
}

\begin{proof} For $\ca=1$ and $\cb=0$, equation \eqref{eq:cplk} reads
$ \CCo\{p_{\ell,k}\}(\lambda)=(2R\pi)^{-{\frac{n}{2}}} \pi 2^{-1} i^\ell\hat f_{\ell,k}(\lambda)
$ $\lambda^{n/2} R J_{\ell+(n-2)/2}(R\lambda)$.
Because $ \hat f_{\ell,k}(\lambda)$ is  continuous,
this implies
$\CCo( p_{\ell,k})(R^{-1} w_{j,\ell}) = 0$ and concludes
the proof.
\end{proof}

As another  application of  Lemma~\ref{lem:reloffandp},
we derive the  following inversion formula.

\begin{cor}\label{cor:inv1} \rot
Let $f \in \mathcal{C}^\infty_0 \left(\R^n \right)$, let $\ca, \cb \in \R$, and consider the
measurement data $g := \Mo_{\ca,\cb} f$  as in \eqref{eq:data}. Then
\begin{multline}\label{eq:inv1}
f(\xx)=\frac{2}{\pi}\sum_{\ell=0}^\infty\sum^{N(n,\ell)}_{k=0} \bkl{\frac{R}{\abs{\xx}}}^{{\frac{n-2}{2}}}
\\ \times \int\half\frac{\CCo\set{g_{\ell,k}}(\lambda)J_{\ell+\frac{n-2}{2}}(\lambda|\xx|) \rmd\lambda}{ ( \ca + \cb R^{-1} \ell)J_{\ell+\frac{n-2}{2}}(R\lambda)-\cb\lambda J_{\ell+\frac{n}{2}}(R\lambda) } Y_{\ell,k}\left(\frac{\xx}{|\xx|}\right).
\end{multline}
\end{cor}

\begin{proof}
Using  \eqref{eq:e1}, \eqref{eq:funkhecke} and writing $( \,\cdot\,)^*$ for the complex conjugation we have
\begin{alignat*}{2}
\hat f_{\ell,k}(\lambda)& =\int_{\sph^{n-1}}\int_{\R^n} e^{-i\lambda\inner{\xx}{\oomega}} f(\xx) \kl{Y_{\ell,k}(\oomega)}^*\rmd\xx \rmd S(\oomega)
\\
&=\int_{\sph^{n-1}} \int\half f(\rho\ttheta)\rho^{n-1} \left(   ~ \int_{\sph^{n-1}} e^{-i\lambda\rho\inner{\ttheta}{\oomega}} \kl{Y_{\ell,k}(\oomega)}^* \rmd S(\oomega)  \right) \rmd\rho \rmd S(\ttheta)\\
&= \int_{\sph^{n-1}} \int\half f(\rho\ttheta)\rho^{n-1} \left( \int_{\sph^{n-1}} e^{i\lambda\rho\inner{\ttheta}{\oomega}} Y_{\ell,k}(\oomega) \rmd S(\oomega)\right)^* \rmd\rho \rmd S(\ttheta) \\
&=  (2\pi)^{\frac{n}{2}}\kl{i^\ell}^*\int_{\sph^{n-1}} \int\half f(\rho\ttheta)\rho^{{\frac{n}{2}}} \lambda^{\frac{2-n}{2}} J_{\ell+\frac{n-2}{2}}(\lambda\rho)\kl{Y_{\ell,k}(\oomega)}^*  \rmd\rho \rmd S(\ttheta)
\end{alignat*}
and therefore
\begin{equation} \label{eq:coeffofff}
\hat f_{\ell,k}(\lambda)
=
(2\pi)^{\frac{n}{2}} \kl{i^\ell}^*  \lambda^{\frac{2-n}{2}}\int\half f_{\ell,k}(\rho)\rho^{{\frac{n}{2}}} J_{\ell+\frac{n-2}{2}}(\lambda\rho)\rmd\rho \,.
\end{equation}
Together with Lemma \ref{lem:reloffandp} this gives
\begin{equation}\label{eq:hinv}
\int\half f_{\ell,k}(\rho)\rho^{{\frac{n}{2}}} J_{\ell+\frac{n-2}{2}}(\lambda\rho)\rmd\rho
=
\frac{2}{\pi}\,  \frac{ R^{\frac{n}{2}} \CCo\set{g_{\ell,k}}(\lambda)}{(R\ca+\cb \ell)
\lambda J_{\ell+\frac{n-2}{2}}(R\lambda)-\cb R \lambda^{2}J_{\ell+\frac{n}{2}}(R\lambda)}\,.
\end{equation}
The left hand side in \eqref{eq:hinv} is recognized as the Hankel transform of order $\ell+(n-2)/2$  of
of the function $\rho \mapsto f_{\ell,k}(\rho)\rho^{(n-2)/2}$. Hence, by applying the inverse Hankel transform, we obtain
\begin{equation*}
f_{\ell,k}(\rho)= R^{\frac{n}{2}} \rho^{-{\frac{n-2}{2}}}
\, \frac{2}{\pi}
\int\half\frac{\CCo\set{g_{\ell,k}}(\lambda)J_{\ell+\frac{n-2}{2}}(\lambda\rho)\rmd\lambda}{(R \ca+\cb \ell)J_{\ell+\frac{n-2}{2}}(R\lambda) - \cb R \lambda J_{\ell+\frac{n}{2}}(R \lambda)    }\,.
\end{equation*}
Together with \eqref{eq:e2},  this gives  the desired result.
\end{proof}

Corollary~\ref{cor:inv1} gives an exact inversion formula for reconstructing
any smooth compactly supported  function  $f \in \mathcal{C}^\infty_0 \left(\R^n \right)$
from data $\Mo_{\ca,\cb} f$.   In order to actually implement the formula one
requires a proper   discretization of the integral.
In  real situations only noisy data $g^{\delta} \in L^2\left(\Sr \times (0, \infty) \right)$ are available and the  cosine transform $\CCo \{ g^{\delta}_{\ell,k}  \}(\lambda)$ will  not vanish at the roots  of the denominator $ (\ca+\cb \ell)J_{\ell+\frac{n-2}{2}}(\lambda)-\cb\lambda J_{\ell+n/2}(\lambda)$ that appears in formula \eqref{eq:inv1}. This means
that Corollary~\ref{cor:inv1} behaves unstable close to the roots  and cannot
 be directly used to reconstruct the initial pressure density $f$.
  In this following section, we avoid the  zeros of the denominator by  using a
  Fourier Bessel series expansion similar as in \cite{haltmeier2007thermoacoustic,zangerl2009exact}.

\section{Stable series inversion formulas}
\label{sec:inv}

{\rot  In practice, reconstructing the initial pressure by the inversion formula stated in Corollary~\ref{cor:inv1} is unstable due to the zeros of $ (\ca+\cb \ell)J_{\ell+\frac{n-2}{2}}(\lambda)-\cb\lambda J_{\ell+n/2}(\lambda)$
in the denominator of   \eqref{eq:inv1}.
To avoid this issue, in this section we derive a series inversion formula
using a Fourier Bessel series,
similar to  \cite{haltmeier2007thermoacoustic,zangerl2009exact}.
We derive  different  inversion formulas for the cases $\cb \neq 0$
and  $\cb=0$, respectively.}
For the following recall that   $w_{j,\ell}$ denotes the $j$-th positive zero of
$J_{\ell+(n-2)/2}$.

The following theorem is the main result of this paper.

\begin{theorem}[Explicit series inversion formula]\label{thm:inv}
Let $f\in \mathcal{C}_0^\infty(\Br)$,
let $\ca, \cb \in \R$, and consider the
measurement data $g := \Mo_{\ca,\cb} f$  as in \eqref{eq:data}.
\begin{enumerate}[label=(\alph*)]
\item  If $\cb \neq 0$, then
\begin{multline}\label{eq:invA}
f(\xx)=   \bkl{ \frac{R}{|\xx|} }^{{\frac{n-2}{2}}}\frac{4}{c_2 \pi}
 \sum_{\ell=0}^\infty
\sum^{N(n,\ell)}_{k=0}
\\
\bkl{ \sum^\infty_{j=1}
  \frac{\CCo\set{g_{\ell,k}}\left(\frac{w_{j,\ell}}{R}\right)}{  w_{j,\ell}^2  J_{\ell+\frac{n}{2}}(w_{j,\ell})^3 }
J_{\ell+\frac{n-2}{2}}\left(\frac{w_{j,\ell}|\xx|}{R}\right)}Y_{\ell,k}\left(\frac{\xx}{|\xx|}\right) \,.
\end{multline}

\item
{\rot If $\cb = 0$, then
\begin{multline}\label{eq:invB}
f(\xx)=  \bkl{ \frac{R}{|\xx|} }^{{\frac{n-2}{2}}}\frac{4}{ \ca R^2 \pi}
 \sum_{\ell=0}^\infty\sum^{N(n,\ell)}_{k=0}
\\
\bkl{  \sum^\infty_{j=1} \frac{\So\left\{ t g_{\ell,k}(t) \right\}\left(\frac{w_{j,\ell}}{R}\right)}{  w_{j,\ell}  J_{\ell+\frac{n}{2}}(w_{j,\ell})^3 }
J_{\ell+\frac{n-2}{2}}\left(\frac{w_{j,\ell}|\xx|}{R}\right)}Y_{\ell,k}\left(\frac{\xx}{|\xx|}\right) \,.
\end{multline} }
\end{enumerate}
\end{theorem}

\begin{proof}
Because   $f_{\ell,k}$
 is of class $C^\infty$ and is compactly supported, we can expand $f_{\ell,k}(\rho)\rho^{{\frac{n-2}{2}}}$  into a Fourier Bessel series \cite{folland2009fourier},\rot
\begin{multline}\label{eq:flkfromFflk1}
f_{\ell,k}(\rho)\rho^{{\frac{n-2}{2}}}
=\sum^\infty_{j=1}\frac{2 \int\half f_{\ell,k}(r)
J_{\ell+\frac{n-2}{2}}\left(\frac{w_{j,\ell}r}{R}\right) r^{{\frac{n}{2}}}\,\mathrm{d}r }{R^2J_{\ell+\frac{n}{2}}(w_{j,\ell})^2} J_{\ell+\frac{n-2}{2}}\left(\frac{w_{j,\ell}\rho}{R}\right)\\
\\ =(2\pi)^{-{\frac{n}{2}}} (i^{-\ell})^*\sum^\infty_{j=1}\frac{2R^{-\frac{n+2}2}}{J_{\ell+\frac{n}{2}}(w_{j,\ell})^2}\hat f_{\ell,k}\left(\frac{w_{j,\ell}}{R}\right)w_{j,\ell}^\frac{n-2}{2} \;
J_{\ell+\frac{n-2}{2}}\left(\frac{w_{j,\ell}\rho}{R}\right) \,,
\end{multline}
where for the second line, we used (\ref{eq:coeffofff}).
From Lemma \ref{lem:reloffandp}, we have
\begin{multline}\label{eq:ww}
\CCo\set{g_{\ell,k}}\left(\frac{w}{R}\right)
=(2R\pi)^{-{\frac{n}{2}}}\pi2^{-1} i^\ell\hat f_{\ell,k}\left(\frac{w}{R}\right)
R^{-{\frac{n}{2}}}w^{\frac{n}{2}} \\ \times
\left[(R\ca+\cb \ell)J_{\ell+\frac{n-2}{2}} ( w) -
\cb w J_{\ell+\frac{n}{2}}(w)\right] \,.
\end{multline}

\begin{itemize}[wide]
\item
For  $\cb \neq 0$ we have $(R\ca+\cb \ell)J_{\ell+\frac{n-2}{2}} ( w_{j,\ell}) -
\cb w_{j,\ell} J_{\ell+\frac{n}{2}}(w_{j,\ell}) = -
\cb w_{j,\ell} J_{\ell+\frac{n}{2}}(w_{j,\ell}) \neq 0$. We can therefore evaluate
\eqref{eq:ww} at $w = w_{j,\ell}$ and solve for
$\hat f_{\ell,k}\left(\frac{w}{R}\right)$. Together with
\eqref{eq:flkfromFflk1}, we obtain
\begin{equation} \label{eq:proof1}
f_{\ell,k}(\rho)\rho^{{\frac{n-2}{2}}}
 = R^{\frac{n-2}{2}}
 \frac{4 }{c_2  \pi}
\sum^\infty_{j=1}\frac{\CCo\set{g_{\ell,k}}\left(\frac{w_{j,\ell}}{R}\right)
J_{\ell+\frac{n-2}{2}}\left(\frac{w_{j,\ell}\rho}{R}\right)}{ J_{\ell+\frac{n}{2}}(w_{j,\ell})^2
w_{j,\ell}^2   J_{\ell+\frac{n}{2}}(w_{j,\ell}) } \,.
\end{equation}

\item
For the case $\cb =0$, equation  \eqref{eq:ww}
gives
$$
\hat f_{\ell,k}\left(\frac{w}{R}\right)=
\frac{\CCo\set{g_{\ell,k}}\left(\frac{w}{R}\right)}{w^{\frac{n}{2}} J_{\ell+\frac{n-2}{2}} ( w)}
\frac{2 (2R\pi)^{{\frac{n}{2}}}  R^{{\frac{n-2}{2}}} }{\ca \pi i^\ell}
\quad \text{ for } w \neq w_{j,\ell}    \,.
$$
For the zeros $w = w_{j,\ell}$ of $J_{\ell+\frac{n-2}{2}} ( w)$
this   is an indeterminate form, which  can be evaluated with L'Hospital's rule.
We have
\begin{multline*}
\partial_w \Bigl( w^{\frac{n}{2}}J_{\ell+\frac{n-2}{2}}(w) \Bigr)
=
\partial_w \Bigl(\lambda^{\ell+n-1}w^{-(\ell+\frac {n-2}2)} J_{\ell+\frac{n-2}{2}}(w) \Bigr)
\\ =(\ell+n-1)w^\frac{n-2}{2} J_{\ell+\frac{n-2}{2}}(w)-w^{\frac{n}{2}}J_{\ell+\frac{n}{2}}(w) \,.
\end{multline*}
Using this and the identity $\partial_w \CCo\set{g_{\ell,k}}(w R^{-1}) = -R^{-1}
\So\set{t g_{\ell,k}(t)}(w R^{-1})$, and applying L'Hospital's rule give
\begin{equation*}%\label{eq:relFf}
\hat f_{\ell,k}\left(\frac{w_{j,\ell}}{R}\right)=
\frac{2 (2R\pi)^{{\frac{n}{2}}}  R^{{\frac{n-2}{2}}} }{\ca \pi i^\ell}
\,
\frac{\So\{ t g_{\ell,k}(t)\}(\frac{w_{j,\ell}}{R})}{R w_{j,\ell}^{\frac{n}2}J_{\ell+\frac{n}{2}}(w_{j,\ell})} \,.
\end{equation*}
Therefore equation  \eqref{eq:flkfromFflk1}  yields
\begin{equation} \label{eq:proof2}
f_{\ell,k}(\rho)\rho^{{\frac{n-2}{2}}}
 =  R^{\frac{n-2}{2}}
 \frac{4 }
{ \ca R^2 \pi}
\sum^\infty_{j=1}\frac{\So\left\{t g_{\ell,k}(t) \right\}\left(\frac{w_{j,\ell}}{R}\right)}
{ w_{j,\ell}   J_{\ell+\frac{n}{2}}(w_{j,\ell})^2 }
J_{\ell+\frac{n-2}{2}}\left(\frac{w_{j,\ell}\rho}{R}\right) \,.
\end{equation}
\end{itemize}
By combing   \eqref{eq:proof1} and  \eqref{eq:proof2}
with the spherical harmonics expansion \eqref{eq:e2},
 we obtain the desired inversion formulas \eqref{eq:invA} and \eqref{eq:invB}.
\end{proof}

{\rot
Note that  the inversion formula \eqref{eq:invA}  for $\Mo_{\ca, \cb}$
might be  slightly surprising  because it holds for any $\cb \neq 0$   regardless of the value of $\ca$.   However, this  independence  is an immediate consequence of the decomposition
$\Mo_{\ca, \cb} f = \ca \Mo_{1, 0} +  \cb \Mo_{0, 1}$, the range condition in Proposition~\ref{cor:range} for the operator $\Mo_{ 1,0}$, and  the particular
form of the right hand side in \eqref{eq:invA}.}
Equation~\eqref{eq:invB}, on the other hand,   is a  new series expansion formula
for the standard PAT forward operator  $\Mo_{1, 0}$.
In the special case $n=2$, it becomes  the  series inversion formula
\cite[Theorem 3.1]{haltmeier2007thermoacoustic}.

\section{Numerical experiments}
\label{sec:num}

In this section, we present reconstruction results with the inversion formulas in
Theorem~\ref{thm:inv} using data $g = \Mo_{\ca,\cb} f$ for different combinations of $\ca, \cb$.
We consider the case of $n=2$ spatial dimensions and use $R=1$.
The 2D case arises in applications where the acoustic pressure is measured by integrating line detectors~\cite{burgholzer2007temporal,paltauf2017piezoelectric}.

% relative data error :   0.4467    0.4472    0.4477
\begin{figure}[htb!]
	\centering
	\includegraphics[width=0.48\textwidth]{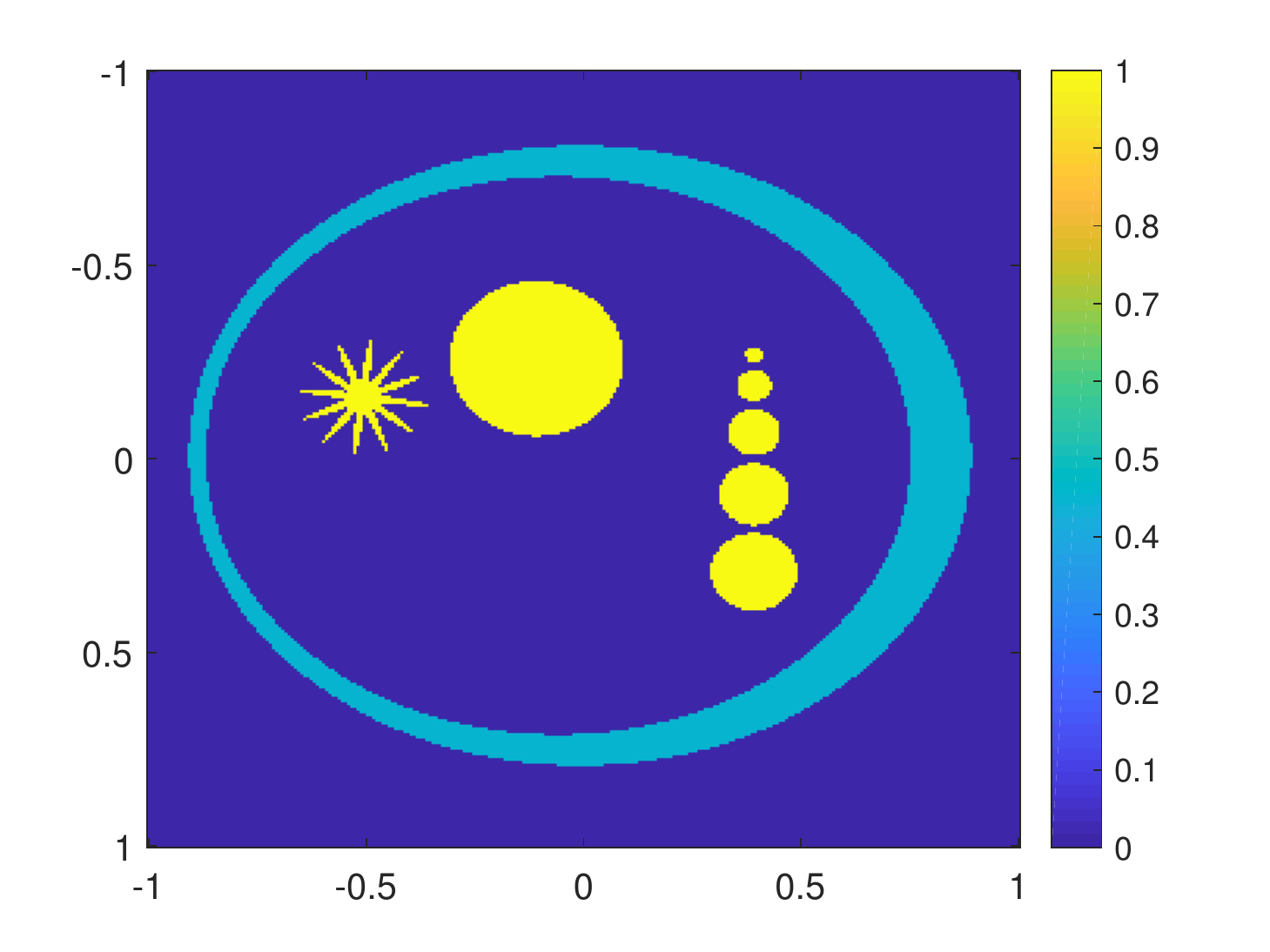}
	\includegraphics[width=0.48\textwidth]{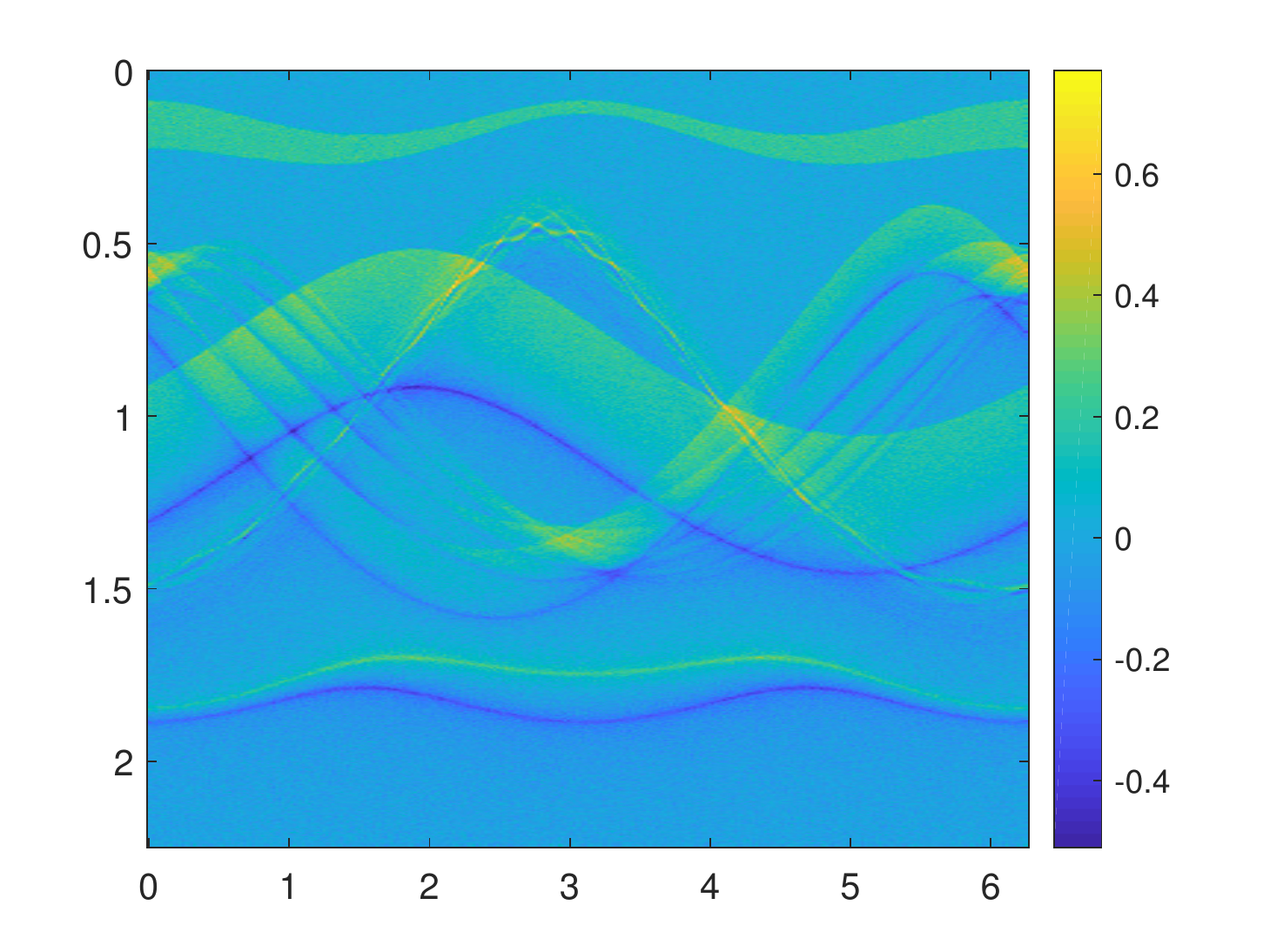}\\
	\includegraphics[width=0.48\textwidth]{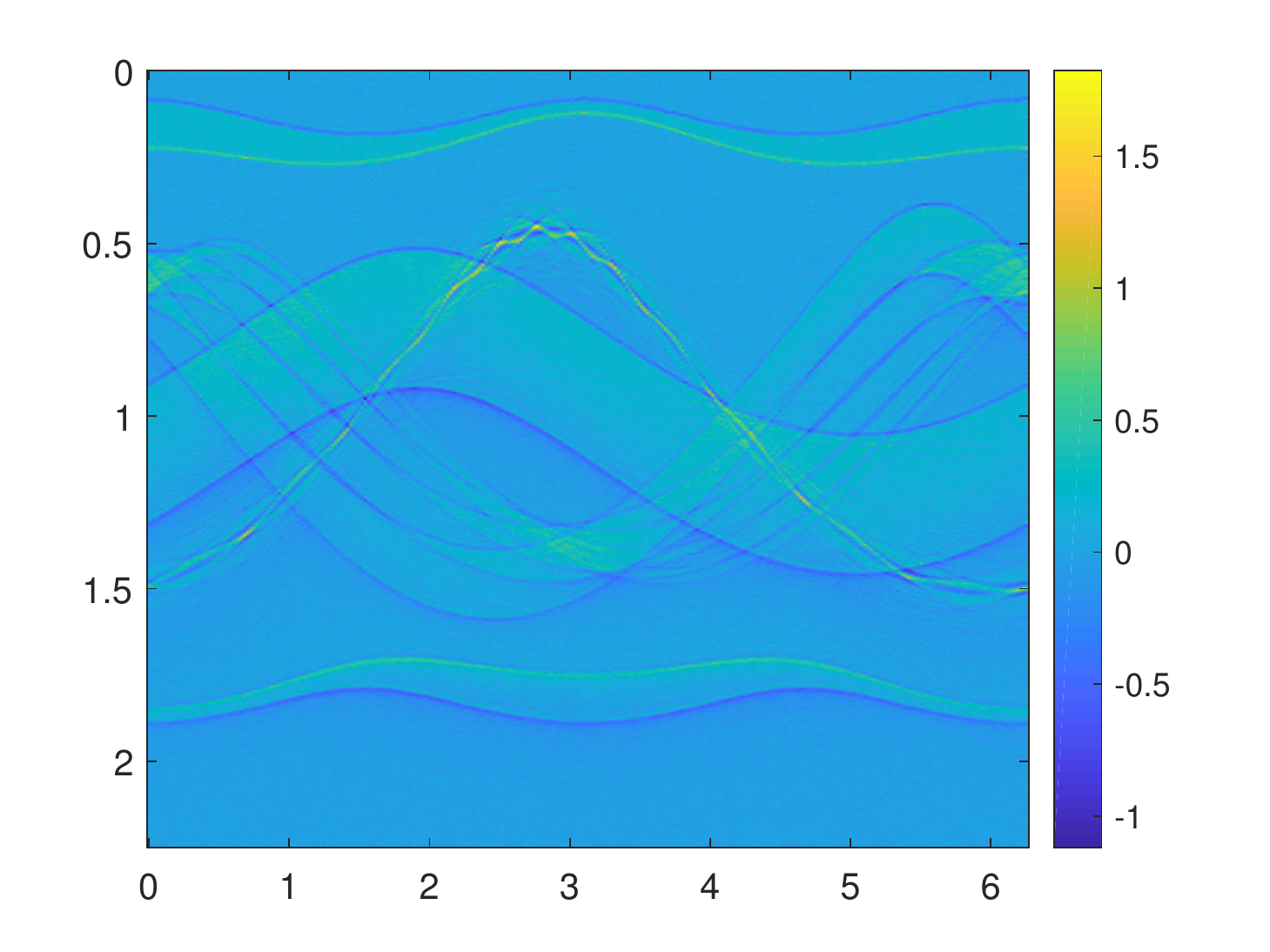}
	\includegraphics[width=0.48\textwidth]{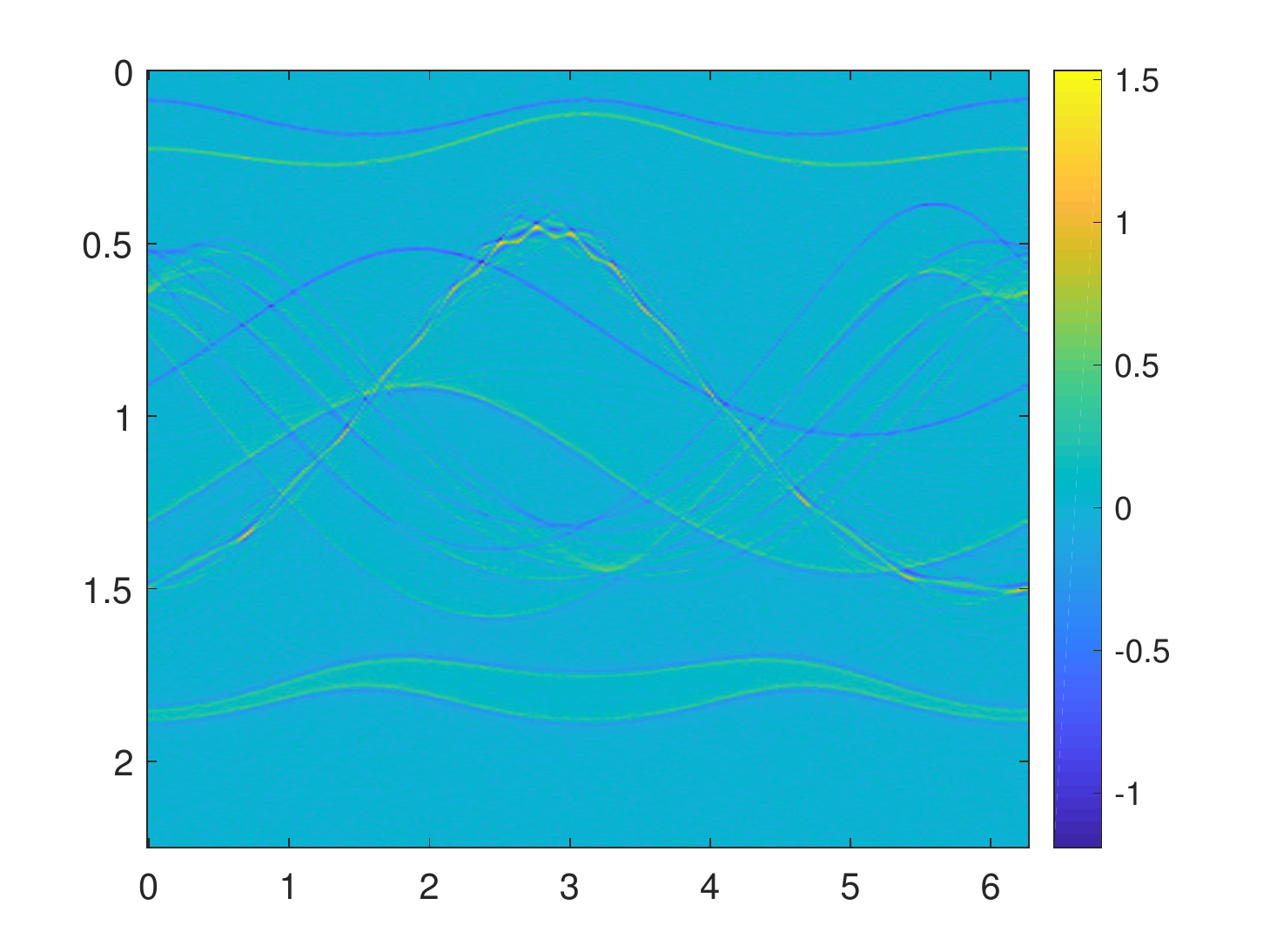}
	\caption{\textsc{Phantom and noisy measurement data.} Top left: Initial pressure $f$ used for the presented numerical results. Top right: Classical PAT data $\Mo_{1,0} f +\xi_{1,0}$. Bottom left: Mixed data $\Mo_{1,1} f +\xi_{1,1}$. Bottom right:
Normal derivative $\Mo_{0,1} f +\xi_{0,1}$.
In all cases Gaussian white noise $\xi_{\ca,\cb}$  with a standard deviation of $\SI{50}{\percent}$ of the $\ell^2$-norm of $\Mo_{\ca,\cb} f$ has been added to the data.}\label{fig:data}
\end{figure}

\subsection{Discretization and data simulation}

In order to compute the pressure field and its gradient restricted to the unit circle,
we use a discrete wave propagation model on a 2D equidistant grid.
To numerically compute the solution at the discretization points we
use the k-space method \cite{cox2007k,mast2001k} implemented as described in  \cite{haltmeier2017analysis}.  Using the k-space method, we obtain the
values of the acoustic pressure field and compute its gradient by symmetric differences.
The pressure and its normal derivative on the circle are obtained by linear interpolation.
For the presented numerical results,  the initial pressure \eqref{fig:data} is given
on  $280 \times 280$ Cartesian grid points in $[-1,1]^2$ and the data $\Mo_{\ca,\cb} f$
 are simulated for $\Nangle = 300$ equidistant sensor locations on unit circle.
 At each sensor we use $\Ntime = 1600$ equidistant time samples in the measurement interval $[0, 6]$.

The  top left image in Figure \ref{fig:data} shows the initial pressure distribution $f$ used our simulations. The top right image shows the classical PAT data $\Mo_{1,0} f+\xi_{1,0}$,
and the bottom right image shows the normal derivative  $\Mo_{0,1} f+\xi_{0,1}$, which has been re-scaled such that $\Mo_{1,0} f$ and $\Mo_{0,1} f$ have the same $\ell^2$-norm. The bottom left image shows the mixed  data $\Mo_{1,1} f +\xi_{1,1}$.
In all cases we have added Gaussian white noise $\xi_{\ca,\cb}$ with a standard deviation of $\SI{50}{\percent}$ of the $\ell^2$-norm $\norm{\Mo_{\ca,\cb} f}_2 :=  (\sum_{\mf, \nf} \abs{\Mo_{\ca,\cb} f[\mf, \nf]}^2 / (\Nangle\Ntime))^{1/2}$, resulting in a relative $\ell^2$-data error $\norm{\xi_{\ca,\cb}}_2/\norm{\Mo_{\ca,\cb} f}_2$ of  $\SI{45}{ \percent}$ in all cases.
Note that we simulated the data until $T = 6$ but only show data until $t=2.2$ in
Figure~\ref{fig:data}, because after $T=2$ the data $\Mo_{\ca,\cb} f$ are smooth
and monotonically decreasing.

\subsection{Implementation of the inversion formulas}

In order to reconstruct the initial pressure distribution, we implement discrete versions of the reconstruction formulas \eqref{eq:invA} and \eqref{eq:invB},
which for $n=2$ and $R  =c_1 = c_2 = 1$ are given by
\begin{align}\label{eq:invA2D}
f(\xx)  &= \frac{4}{\pi \sqrt{2\pi} } \sum_{ k \in \mathbb{Z}  }
\bkl{\sum_{j = 1}^\infty        \frac{   J_{k}
	\left(  \omega_{j,k}  \rho \right) }{ \omega^2_{j,k } J_{k+1}\left(  \omega_{j,k}     \right)^3 }   \CCo \left\{  g_k \right\}\left( \omega_{j,k} \right) }  e^{ik \varphi}
	\\ \label{eq:invB2D}
f(\xx)  &=  \frac{4}{\pi\sqrt{2\pi}} \sum_{ k \in \mathbb{Z}  }  \bkl{ \sum_{j = 1}^\infty  \frac{   J_{k}
	\left(  \omega_{j,k}  \rho \right) }{\omega_{j,k}  J_{k+1}\left(  \omega_{j,k}     \right)^3 }
	    \So \left\{ t  g_k \right\}\left( \omega_{j,k} \right) }
	e^{ik \varphi} \,.
\end{align}
Here $g_k(t)  = \frac{1}{\sqrt{2\pi} } \int_{0}^{2\pi} g(\theta(\varphi) , t) e^{-ik\varphi} \rmd \varphi$
are the Fourier coefficients in the angular variable, and
$ \xx = \rho \theta(\varphi) $ with
$\theta(\varphi) =(\cos \varphi, \sin \varphi)$ and $\varphi \in [0, 2\pi)$.
Moreover, $\omega_{j,k}$ are the positive roots of $J_k$.
Formulas \eqref{eq:invA2D} and \eqref{eq:invB2D}
are implemented  following  \cite{haltmeier2009frequency}
where the inversion formula \eqref{eq:invB2D} for standard PAT
data are considered.

% exact data
% error p-formula  0.2690    0.4148    0.9622
% error np-formula 1.1829    0.3547    0.2561
\begin{figure}[htb!]
	\centering
	\includegraphics[width=0.48\textwidth]{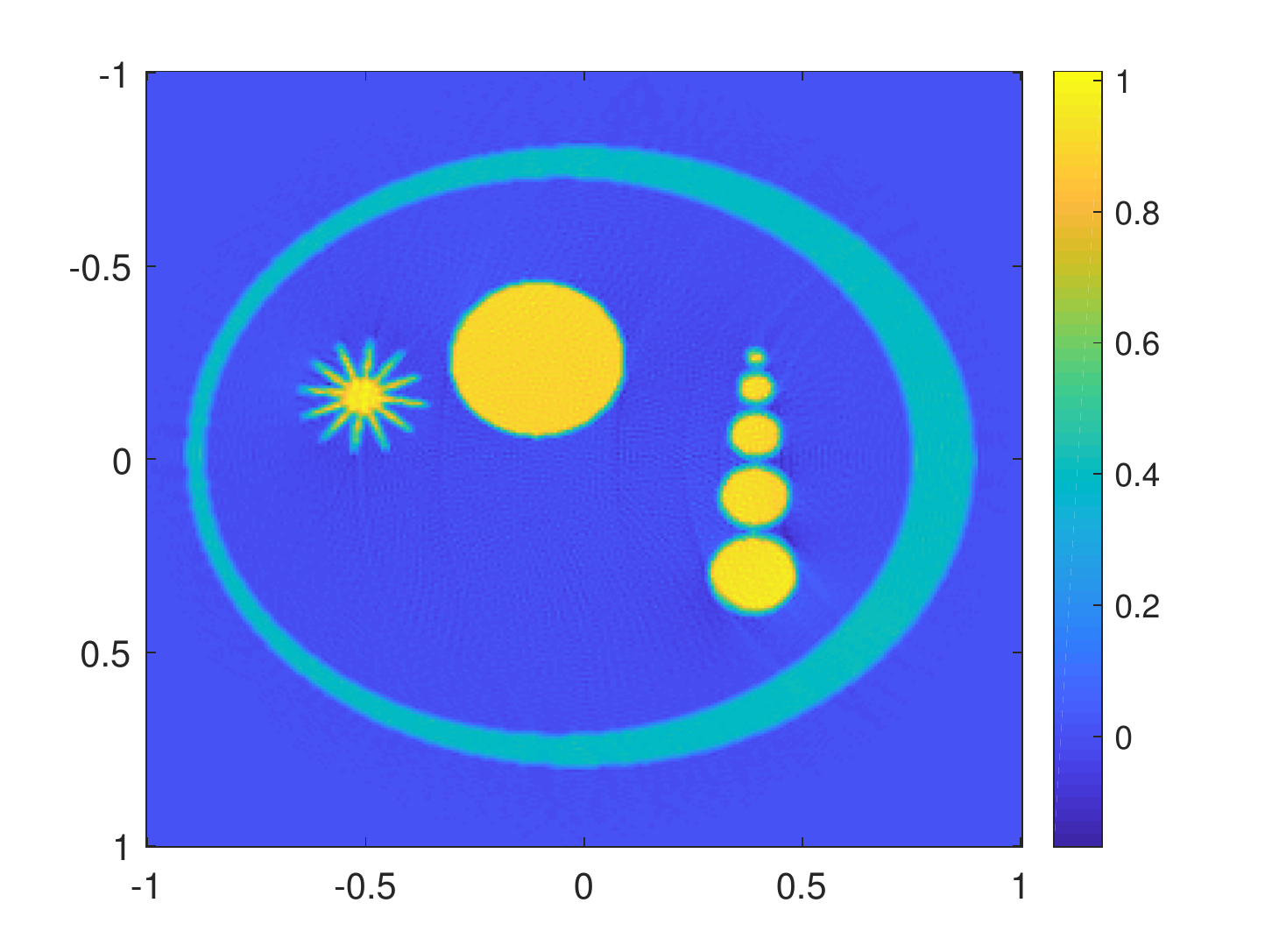}
	\includegraphics[width=0.48\textwidth]{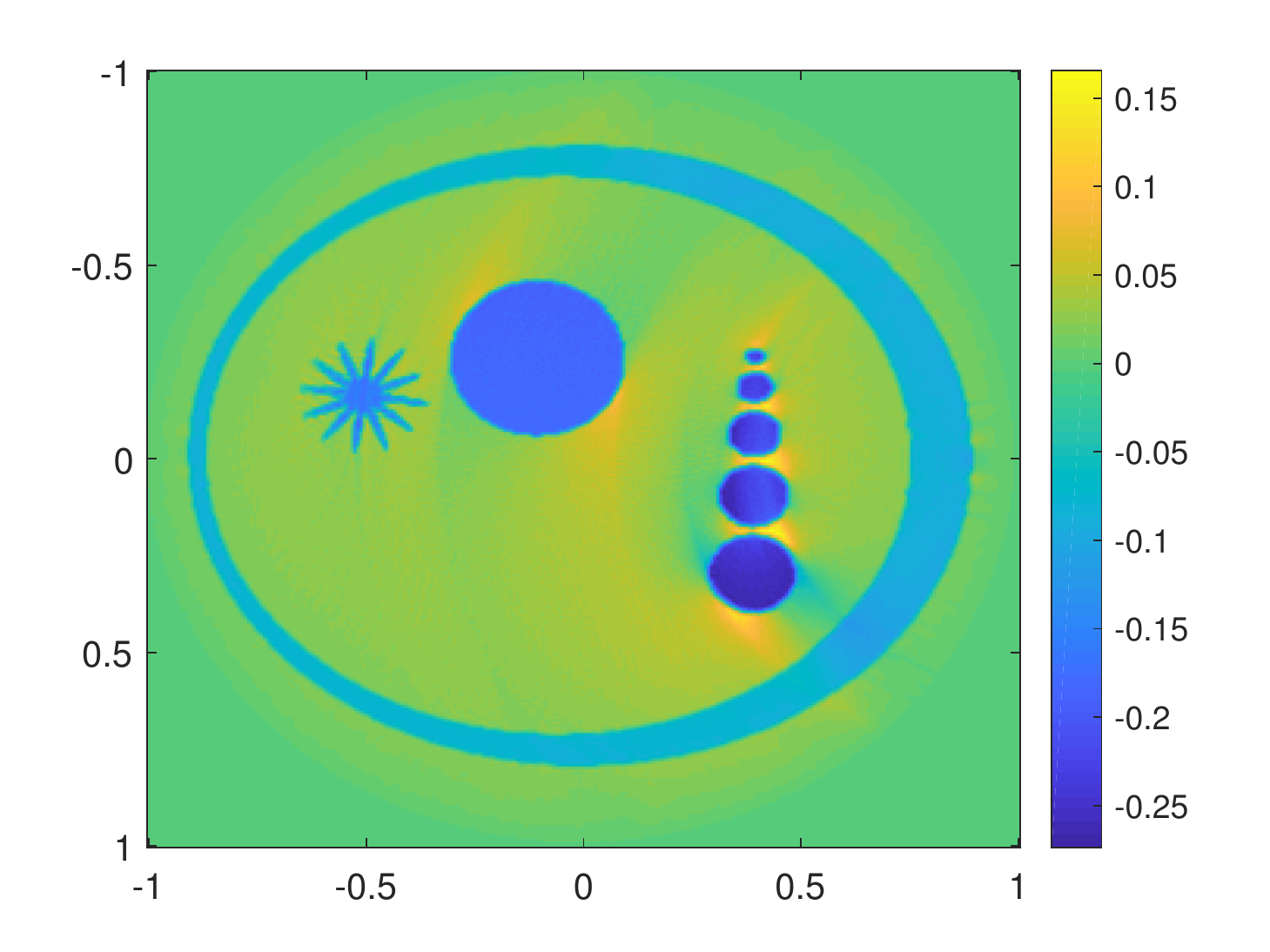}\\
	\includegraphics[width=0.48\textwidth]{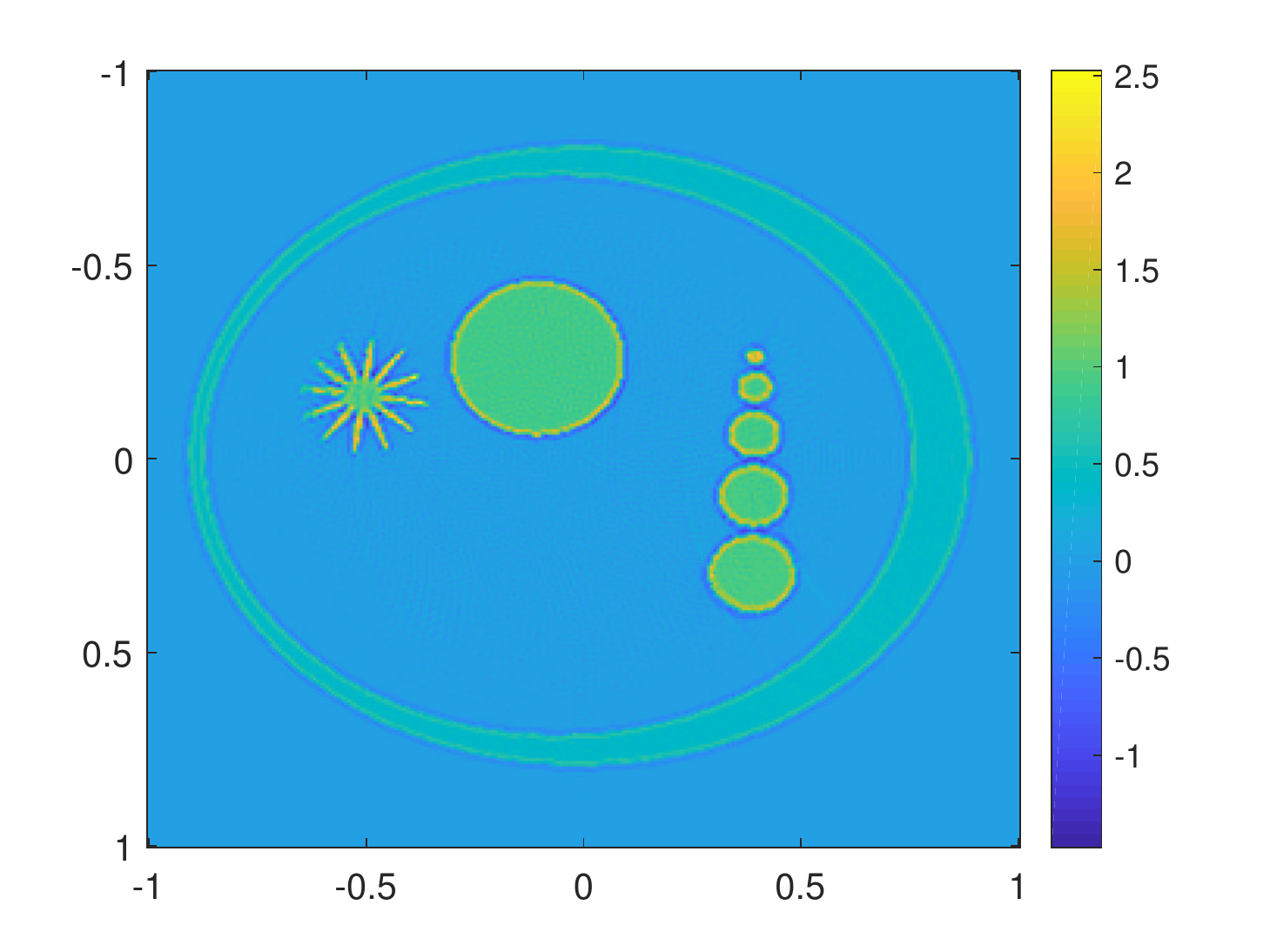}
	\includegraphics[width=0.48\textwidth]{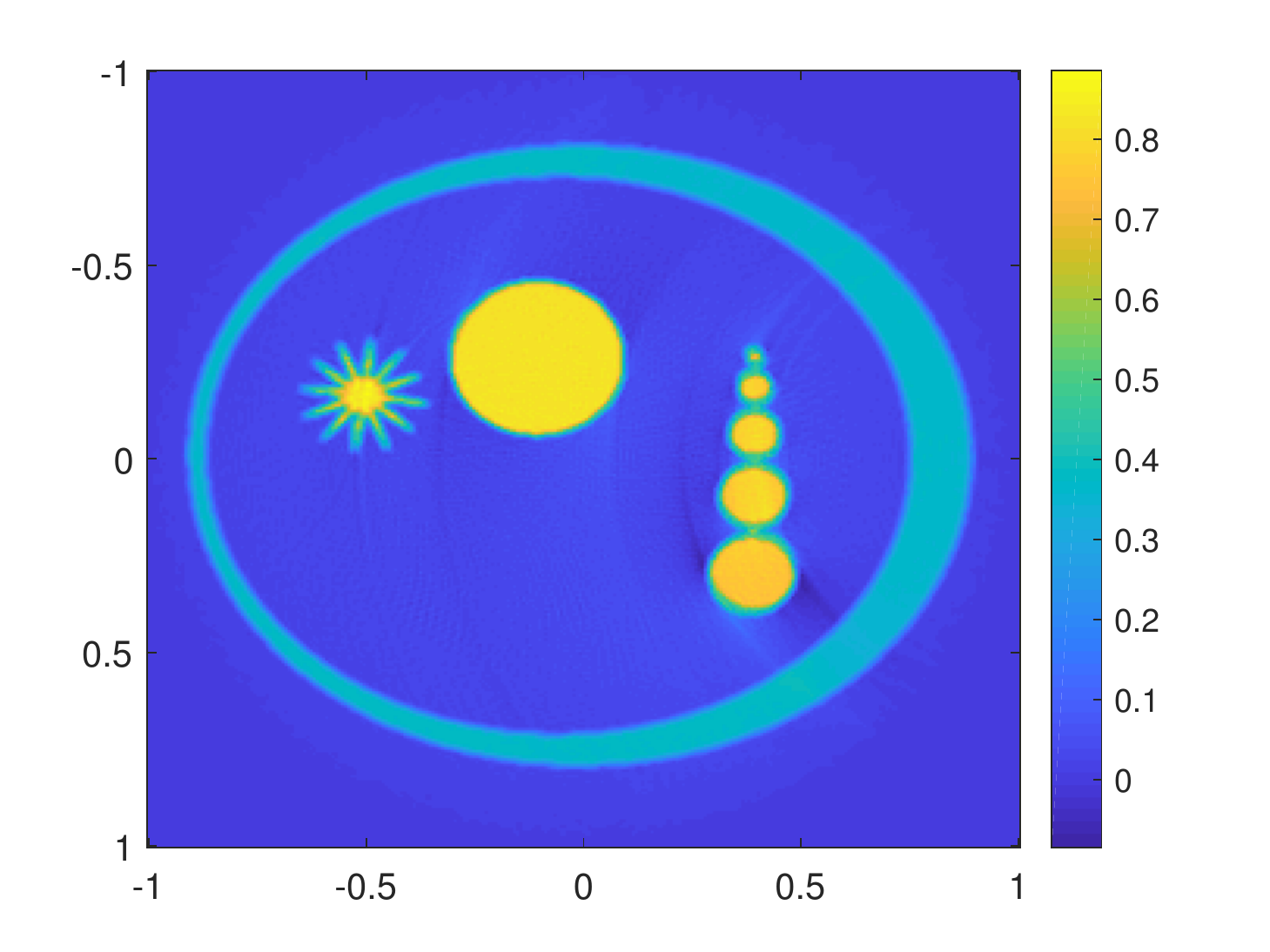}\\
	\includegraphics[width=0.48\textwidth]{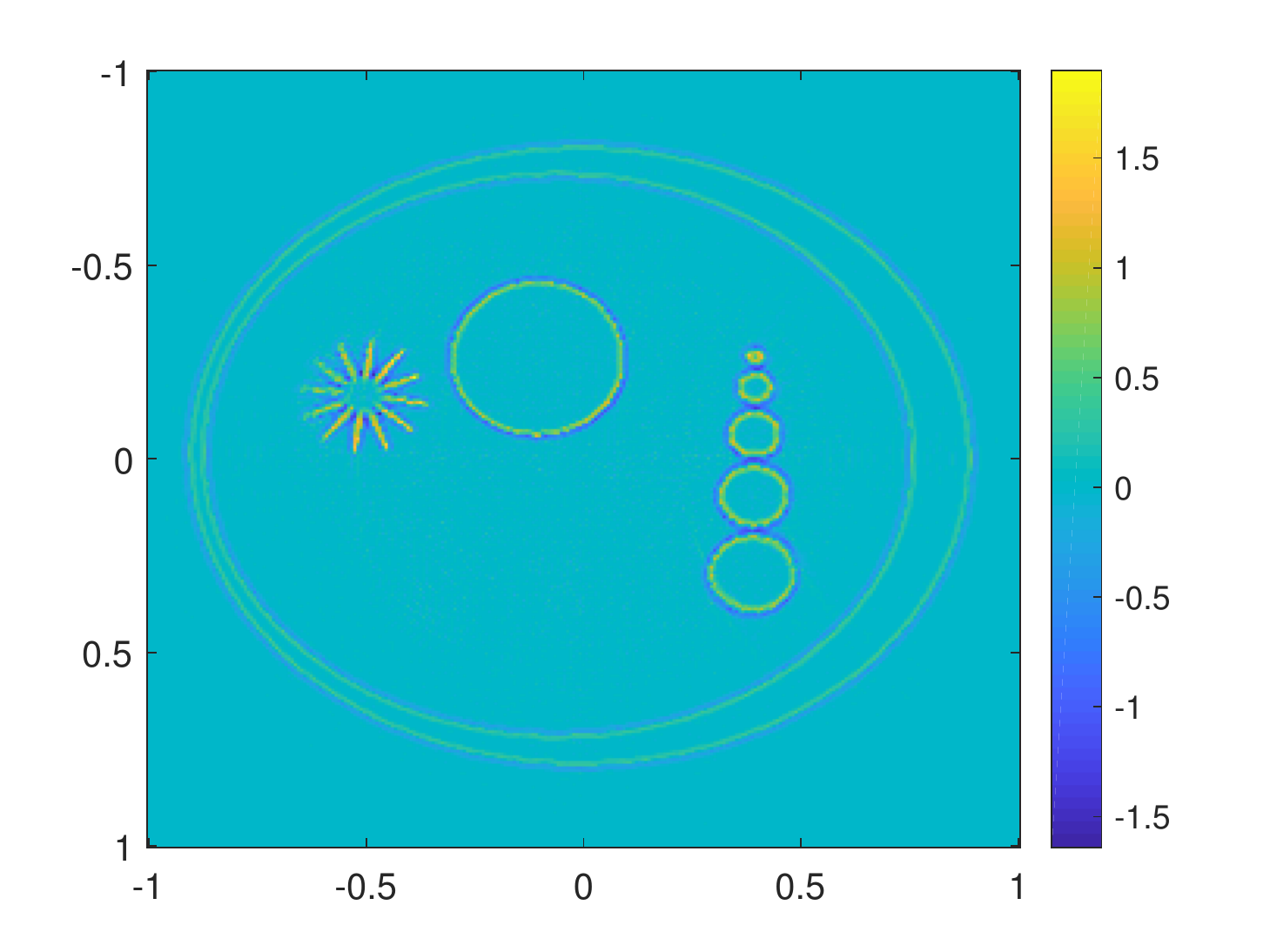}
	\includegraphics[width=0.48\textwidth]{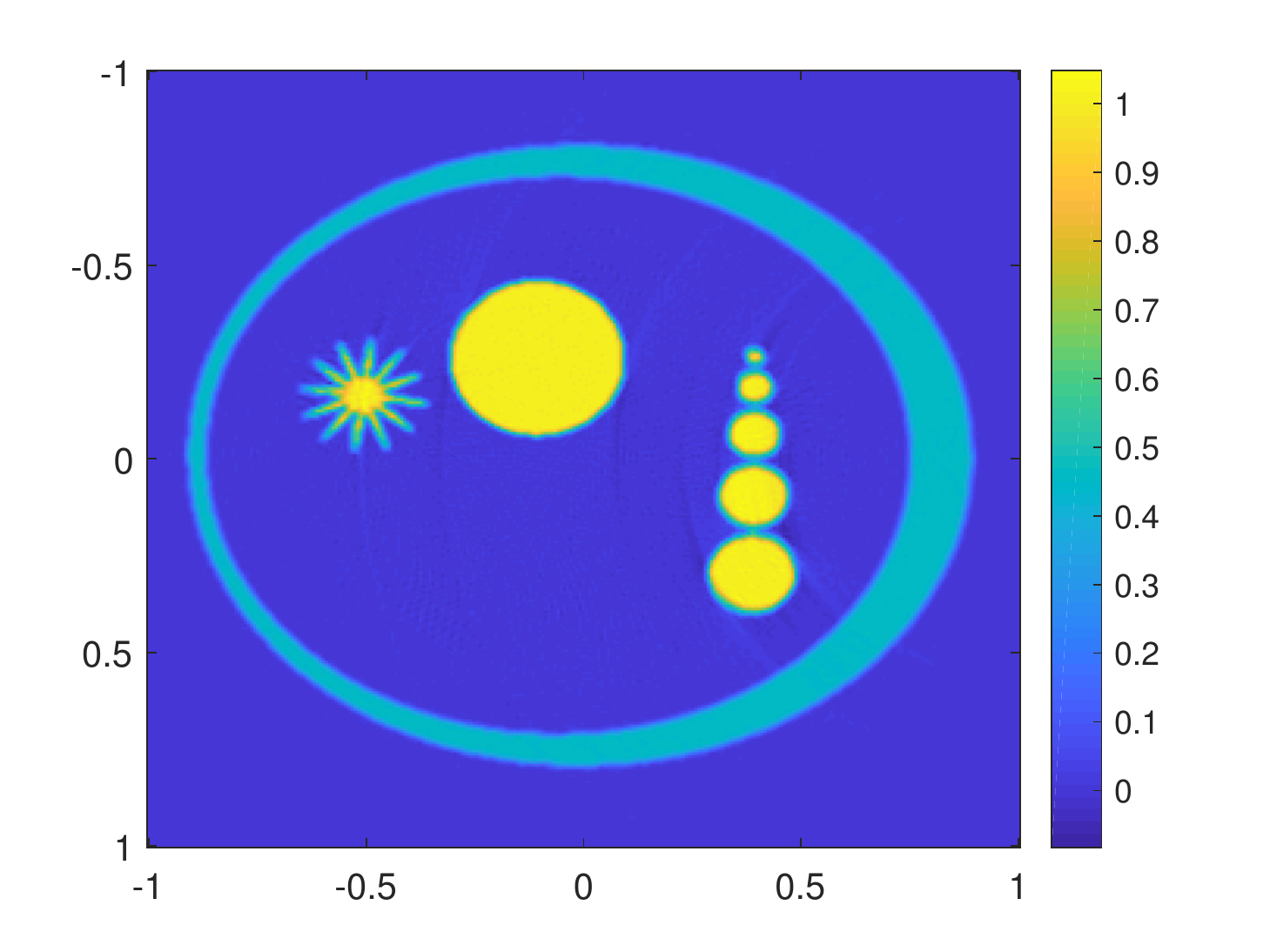}
	\caption{\textsc{Exact data reconstruction:}
The  left column corresponds to reconstructions with formula
	\eqref{eq:invB} (exact for $\cb= 0$)  and the  right column to    reconstructions with formula
	\eqref{eq:invA} (exact for $\cb\neq  0$).
	Top: Standard PAT  data $\Mo_{1,0} f$. Middle: Mixed data $\Mo_{1,1} f$.
	Bottom: Normal derivative  $\Mo_{0,1} f$. Only for the  top left, middle right and
	bottom right images the applied inversion formula matches the data.
	}\label{fig:exact}
\end{figure}

Below we briefly describe  the discretization of  inversion formula  \eqref{eq:invA2D}.
The numerical reconstruction uses discrete samples
$  \gdata \left[ \mf, \nf   \right]  = g\left( \theta_\mf, t_{\nf}  \right)$
for $ (\mf,  \nf  ) \in \left\{1, \dots, \Nangle \right\} \times   \left\{1, \dots, \Ntime \right\}$
where  $\theta_{\mf} = 2 \pi \left( \mf- 1\right)/\Nangle$ and $\nf = T \left( \nf -1 \right)/\Ntime$.
First,  the Fourier coefficients $\gdata_k$ in the angular variable are computed by the FFT algorithm.
Next, the cosine transform is approximated by $
\Cos \{ \gdata_k \}[j] := \frac{T}{\Ntime} \sum_{\nf = 1 }^{\Ntime}    \gdata_k [ \nf ] \cos  ( \omega_{ j, k  } t_{\nf}  )$. This is implemented by matrix-vector
multiplication where the entries $\cos  ( \omega_{ j, k  } t_{\nf}  )$ are pre-computed and stored.  Finally, we evaluate
\eqref{eq:invA2D} by truncating both sums.
Formula  \eqref{eq:invB2D} is discretized in an analogous manner.
This results in  the discrete versions of  \eqref{eq:invA2D},  \eqref{eq:invB2D}
\begin{align}\label{eq:inv-d}
f(\xx)  &= \frac{4}{\pi \sqrt{2\pi}} \sum_{ k = - \Nangle/2  }^{\Nangle/2 -1 }  ~
\bkl{  \sum_{j = 1}^{\Nrad}        \frac{    J_{k}
 	\left(  \omega_{j,k}  \rho \right) }{\omega^2_{j,k } J_{k+1}\left(  \omega_{j,k}    \right)^3 }   \Cos \{ \gdata_k \}[j]  }
	e^{ik \varphi}
	\\ \label{eq:inv-d2}
f(\xx)  &=  \frac{4}{\pi \sqrt{2\pi}} \sum_{ k = - \Nangle/2  }^{\Nangle/2 -1 }  ~
\bkl{  \sum_{j = 1}^{\Nrad}        \frac{    J_{k}
 	\left(  \omega_{j,k}  \rho \right) }{\omega_{j,k } J_{k+1}\left(  \omega_{j,k}    \right)^3 }   \Sin  \left\{  t \gdata_k \right\} [j]  }
	e^{ik \varphi}	
	\,.
\end{align}
In both formulas, the  inner sums  are  evaluated by matrix-vector   multiplications where the required matrix  elements are pre-computed and stored.
The outer sums are  evaluated  with  the inverse FFT algorithms.
Application of \eqref{eq:inv-d} (or \eqref{eq:inv-d2}) with a standard Matlab
implementation on a desktop PC with $\SI{16}{\giga\byte}$ RAM  and
$\SI{3.40}{\mega\hertz}$ eight-core processor with $\Nangle =300$, $\Nrad=180$ and $\Ntime=1200$ takes about $\SI{1.2}{\sec}$.
{\rot Note that $\omega_{j,k}$ are the positive roots of $J_k$ and therefore
 $J_{k+1} (  \omega_{j,k}   ) \neq 0 $ which implies that neither \eqref{eq:inv-d} nor \eqref{eq:inv-d2} suffer from a division by zero problem. }

% noisy data
% error p-formula 0.3150    0.4732    0.9761
% error np-formula  1.1866    0.3416    0.2439
\begin{figure}[htb!]
	\centering
	\includegraphics[width=0.48\textwidth]{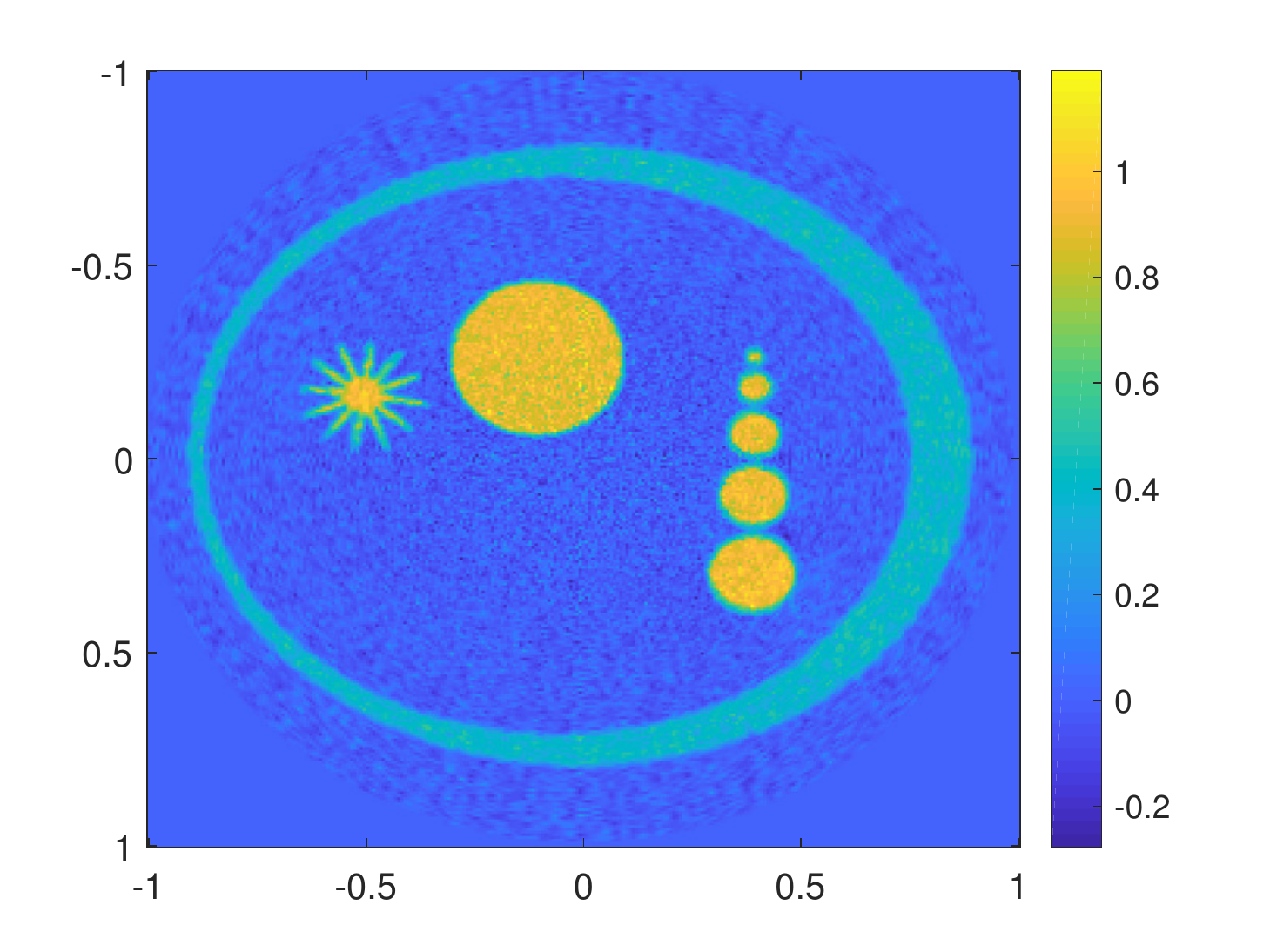}
	\includegraphics[width=0.48\textwidth]{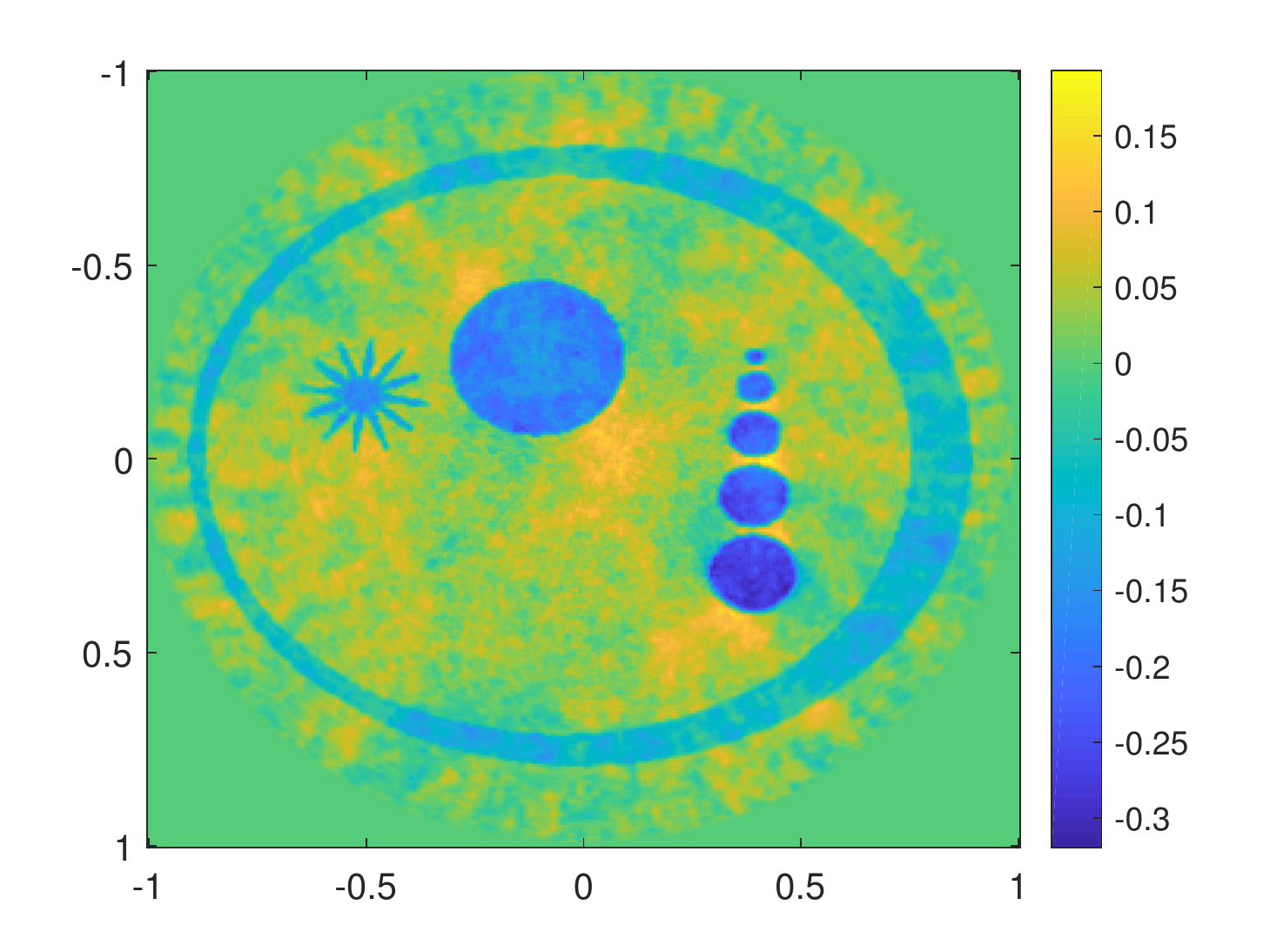}\\
	\includegraphics[width=0.48\textwidth]{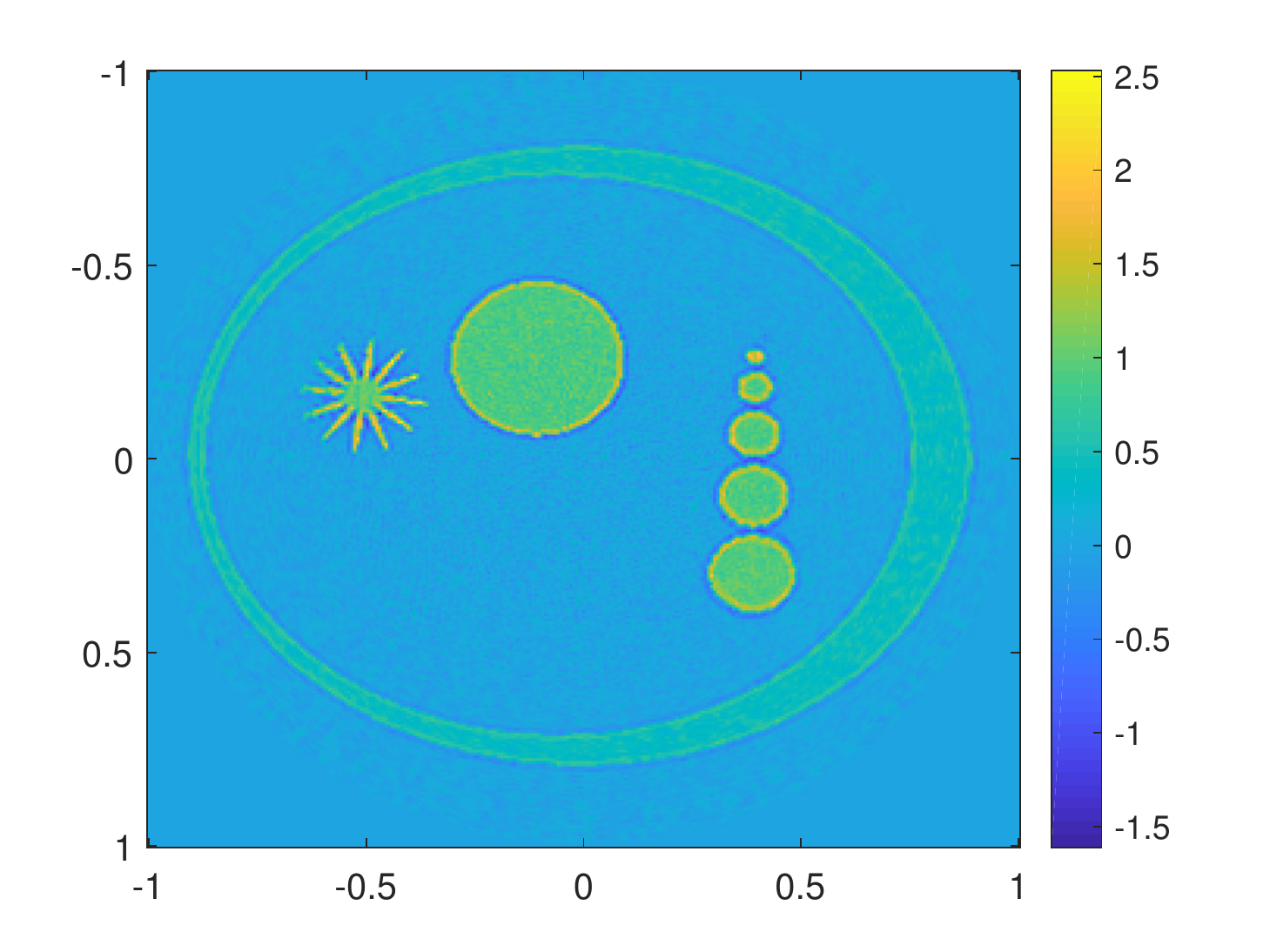}
	\includegraphics[width=0.48\textwidth]{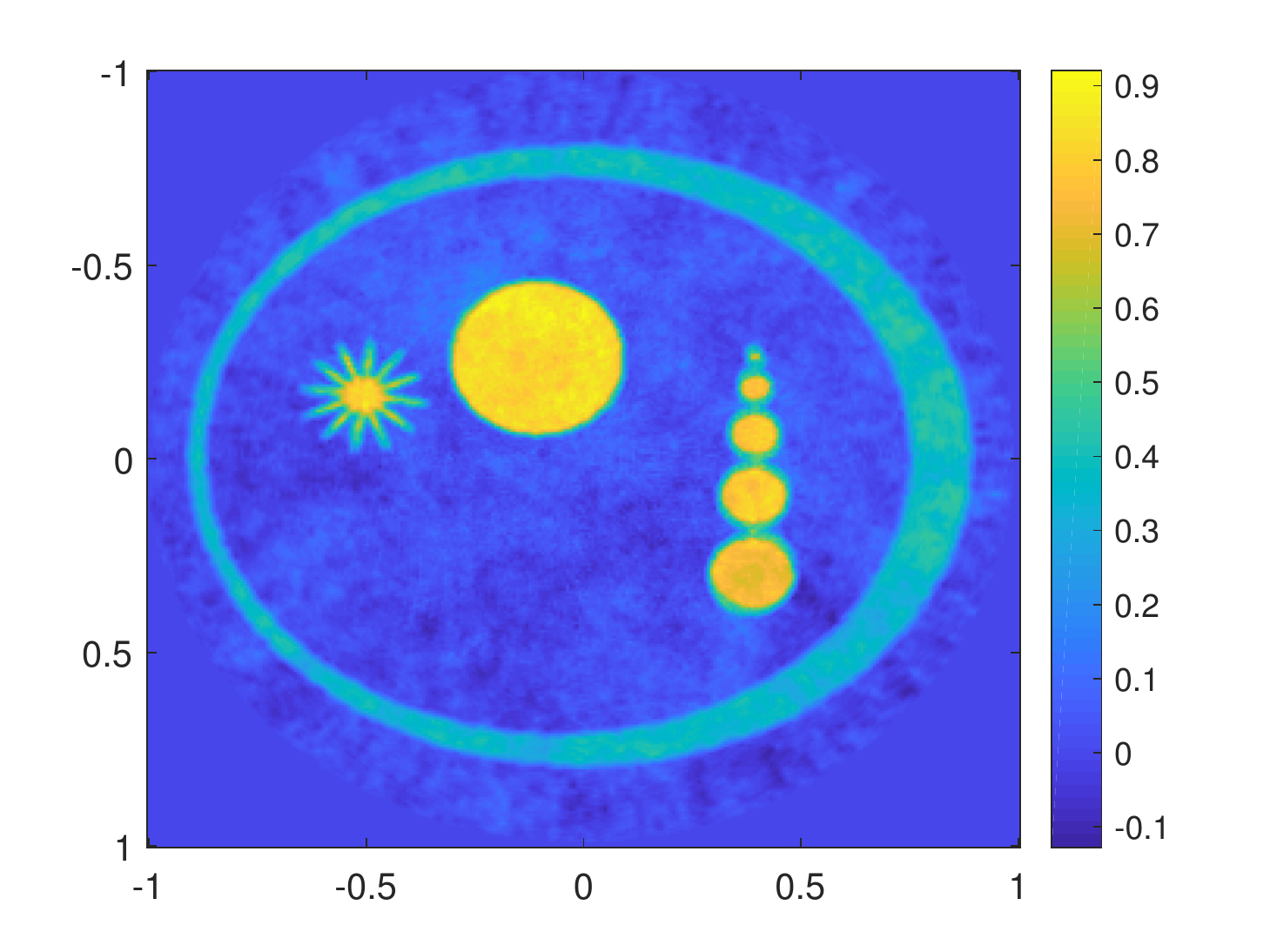}\\
	\includegraphics[width=0.48\textwidth]{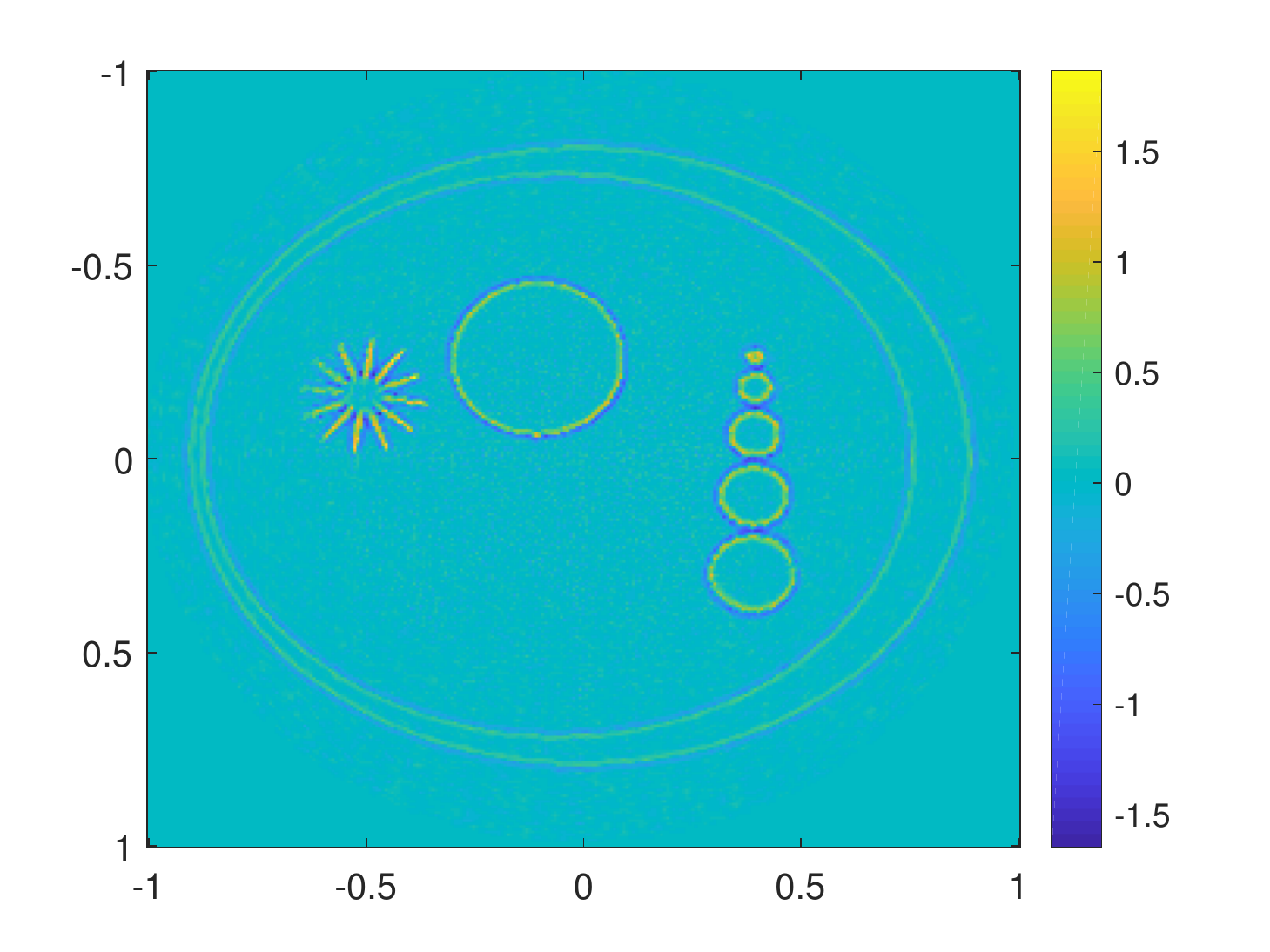}
	\includegraphics[width=0.48\textwidth]{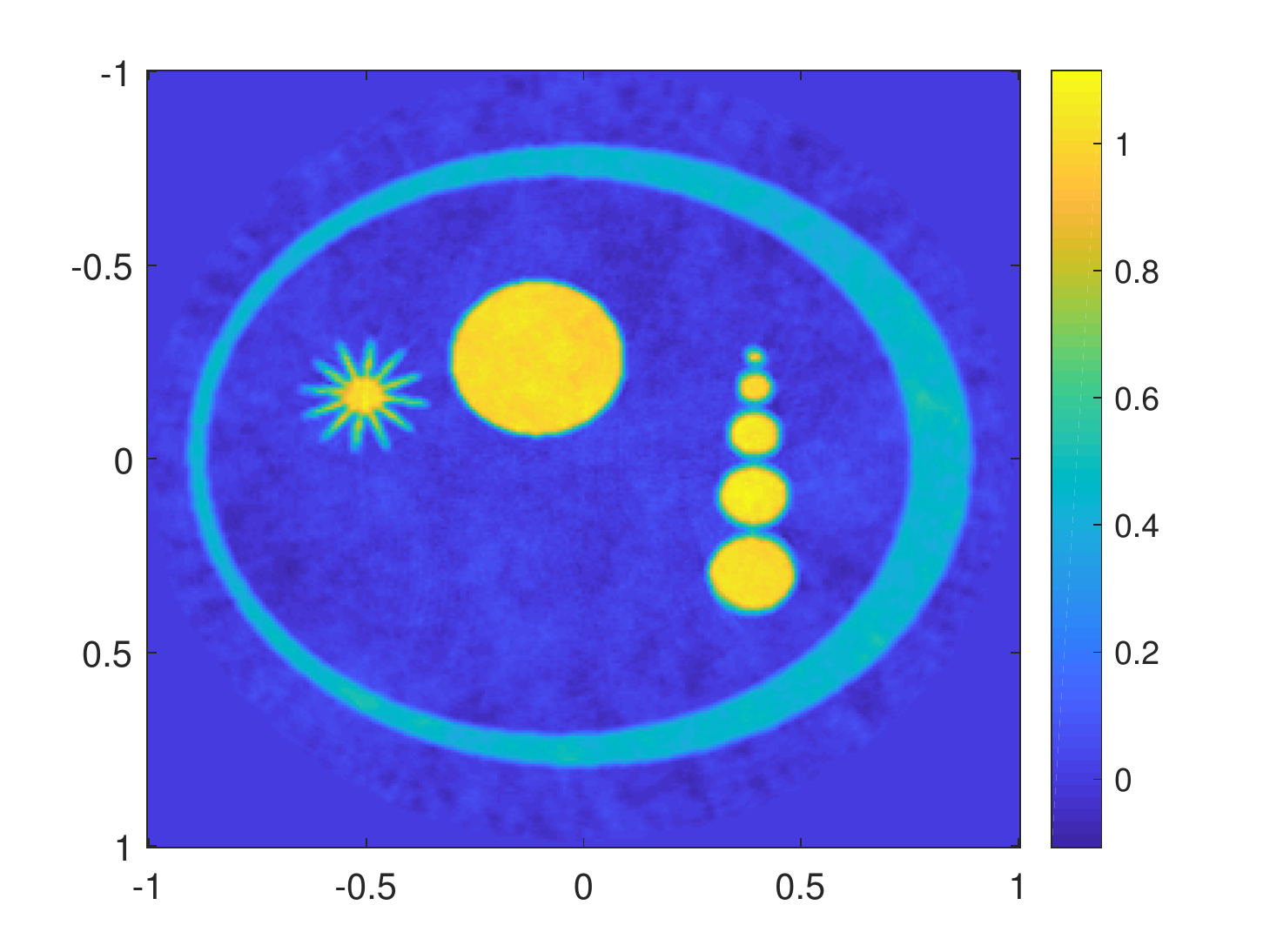}
	\caption{\textsc{Noisy data reconstruction:}
	The  left column corresponds to reconstructions with formula
	\eqref{eq:invB} (exact for $\cb= 0$)  and the  right column to    reconstructions with formula
	\eqref{eq:invA} (exact for $\cb\neq  0$).
	Top: Standard PAT  data $\Mo_{1,0} f +\xi_{1,0}$. Middle: Mixed data $\Mo_{1,1} f
+ \xi_{1,1}$.
	Bottom: Normal derivative  $\Mo_{0,1} f + \xi_{0,1}$.
Only for the  top left, middle right and
	bottom right images the applied inversion formula matches the data.
	}\label{fig:noisy}
\end{figure}

\subsection{Reconstruction results}
{ \rot

Figure \ref{fig:exact} shows  reconstruction results for the simulated data without noise using both formulas \eqref{eq:inv-d},
\eqref{eq:inv-d2}  applied to all three data cases $\Mo_{1,0}f$ (top), $\Mo_{1,1}f$ (middle)
and $\Mo_{0,1}f$ (bottom).   The left  column show the  reconstruction with
 $c_2=0$ formula   and the right column  the results with the $c_2 \neq 0$ formula.
 Note that up to discretization error and truncation error at time $T=6$,
Theorem~\ref{thm:inv}  shows that \eqref{eq:inv-d}  is exact
for $\Mo_{1,0}f$, and  that \eqref{eq:inv-d}  is exact
for $\Mo_{1,1}f$ and $\Mo_{0,1}f$. The numerical results shown in Figure
\ref{fig:exact} support these theoretical findings.
Reconstruction results for noisy  are shown in Figure~\ref{fig:noisy},
which again support exactness of inversion formula and further shows
their stability  with respect to Gaussian noise.
A quantitative error analysis is shown in Figure~\ref{fig:err},
where the relative reconstruction error for cases is shown
in dependence of the noise level of the data.

\begin{figure}[htb!]
	\centering
	\includegraphics[width=0.48\textwidth]{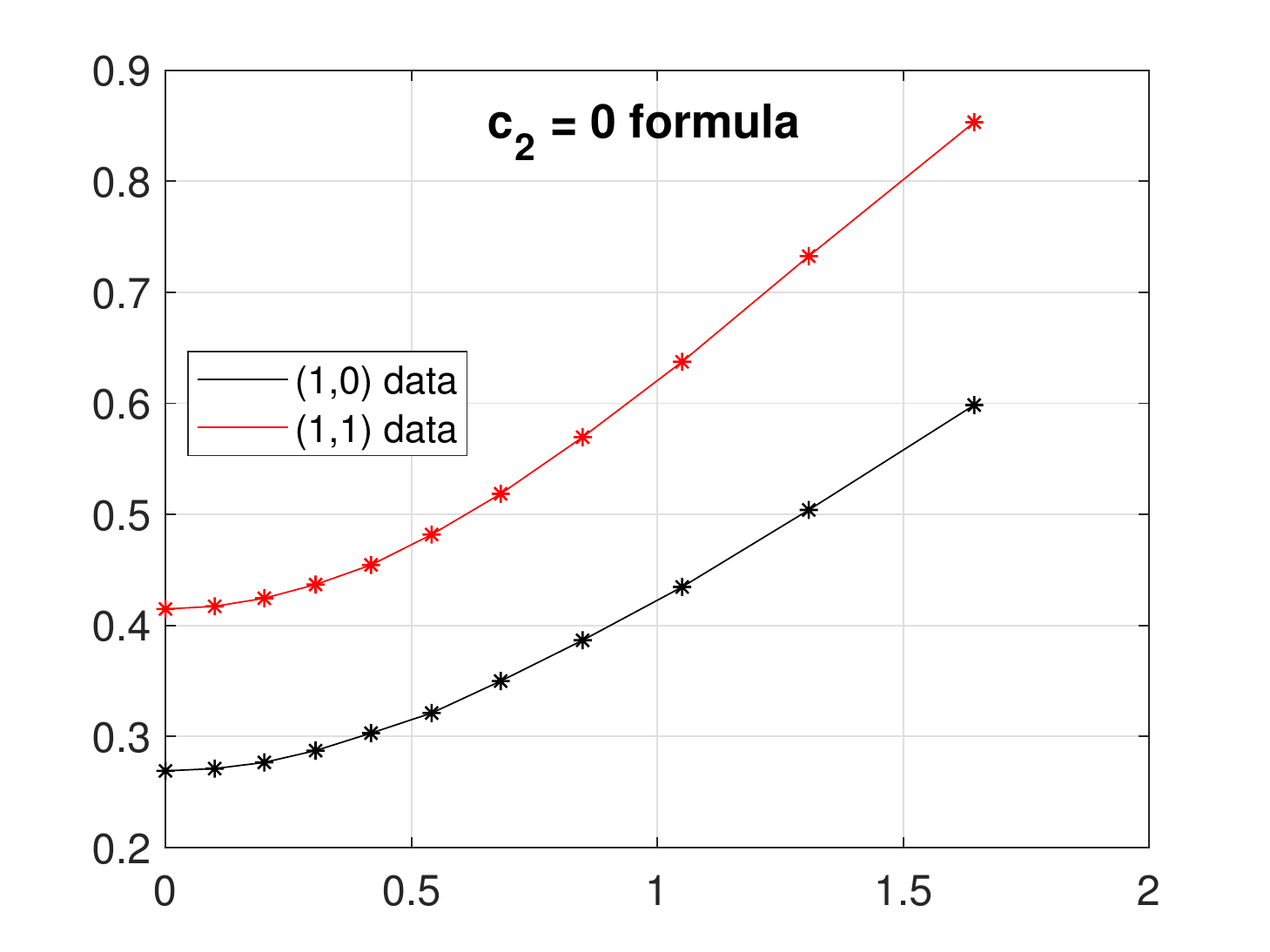}
	\includegraphics[width=0.48\textwidth]{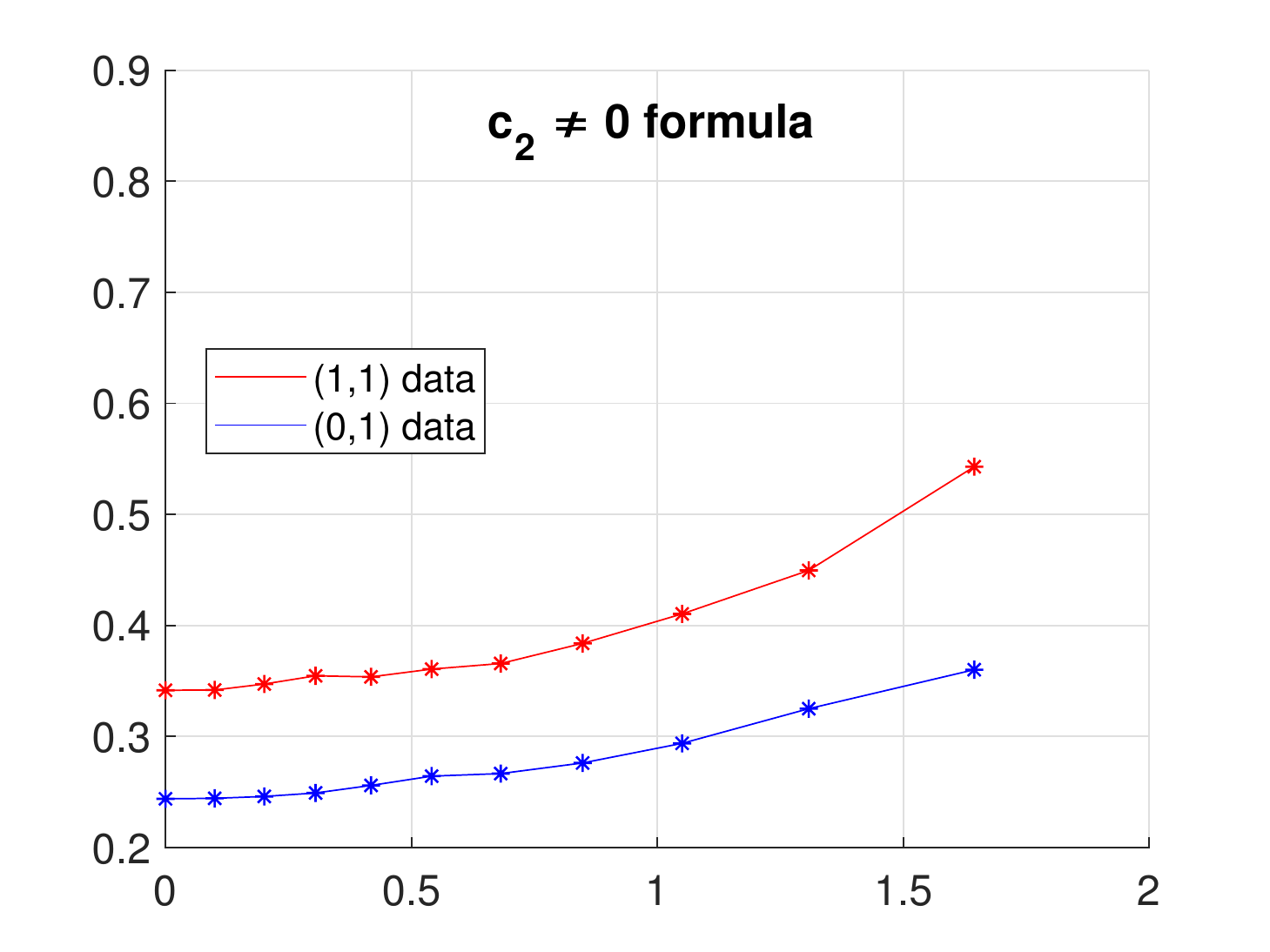}
	\caption{\rot \textsc{Quantitative error analysis:}
	The plots show the relative reconstruction  error
	$\norm{f-f_{\rm rec}}_2/\norm{f}_2$ (vertical axis)
	versus the relative data error
	$\norm{g - g^\delta}_2/\norm{g}_2$ (horizontal axis). Note that the
    axis are not in percent, such that  value 1 on the axis corresponds to
	$\SI{100}{\percent}$. Left: results for formula \eqref{eq:invB}.
	Right: results for formula \eqref{eq:invA}.  }\label{fig:err}
\end{figure}

}
\subsection{Discussion}

First, note that the theoretically exact inversion formulas \eqref{eq:invA} and \eqref{eq:invB} use data for all times whereas the discrete  counterparts  only use data up to finite time.
Nevertheless,  the discrete formula \eqref{eq:inv-d} applied to $\Mo_{1,0}f$ and \eqref{eq:inv-d2} applied to $\Mo_{0,1}f$ give visually satisfactory results.
This  is consistent with our previous observations \cite{burgholzer2007temporal,FinHalRak07,haltmeier2007thermoacoustic}.
However,  the  relative reconstruction  error  is quite large, even for
data without noise.   We mainly  address this due to the required truncation
of the data.  In future work we will therefore address  this issue and aim deriving  exact inversion  formula  which only use data until a finite time $T$.x

Second, up to discretization and truncation errors, due to the range condition  in Corollary \ref{cor:range},  the inversion formula  \eqref{eq:inv-d} applied    to classical PAT data  $\Mo_{1,0} f$ should yield the zero image.
From Figures \ref{fig:exact} (top right) and \ref{fig:noisy} (top right)
 we see that this is clearly  not the case. Again, this is mainly due to the data
 truncation, which we have  verified (result not shows) by varying $T$.
 The introduced violation of the range condition is also the reason that
 the   inversion formula applied to $\Mo_{1,1} f$ yields worse results
 then the result  for  $\Mo_{0,1} f$.  {\rot Finally, applying the inversion formula \eqref{eq:invB}
 to data $\Mo_{1,1} f$ results in amplified boundary structures.
 Visually the amplification is quite appealing which may be the reason that
 commonly data are modeled by $\Mo_{1,0} f$ instead of the more general
 model $\Mo_{\ca,\cb} f$. Quantitative analyzing the effect of $\Mo_{0,1} f$
 component to \eqref{eq:invB} is an interesting open issue. }

\section{Conclusion}
\label{sec:conclusion}

We investigated  PAT with the  direction dependent data model \eqref{eq:data},
which uses linear combinations $\Mo_{\ca, \cb} f  = \ca \Mo_{1,0} f + \cb
\Mo_{0,1} f$ of the acoustic pressure and its normal derivative. We developed an
 exact and  stable reconstruction formula for the special case of spherical detection geometry. Numerical results show the validity of the proposed approach. Investigating   such data for more general detection geometries in PAT is an interesting line of future research.
 Moreover, deriving an explicit inversion formula  that only used  data  for
 $t \leq T < \infty$ is an important future aspect.
 Finally, in  future work we will test our
 formulas on experimental data and investigate which values of  $\ca$, $\cb$
 actually are the best to accurately model experimental PAT data,
 for example using small piezoelectric sensors or piezoelectric
 line sensors.

\section*{Acknowledgments}
The work of G.Z. and M.H. has been supported by the Austrian Science Fund (FWF),
project P 30747-N32. The work of S.M. has been supported by the
National Research Foundation of Korea grant funded by the Korea
government (MSIP) (NRF-2018R1D1A3B07041149).

%\bibliographystyle{plain}
%\bibliography{wave_op}

\end{document}